\documentclass[11pt]{amsart} \textwidth=14.5cm \oddsidemargin=1cm
\usepackage{amsmath,wasysym}
\usepackage{amsxtra}
\usepackage{amscd}
\usepackage{amsthm}
\usepackage{amsfonts}
\usepackage{amssymb}
\usepackage{eucal}
\usepackage{graphics, color}
\usepackage{hyperref,mathrsfs}
\usepackage[usenames,dvipsnames]{xcolor}

\input prepictex
\input pictex
\input postpictex

\newtheorem{thm}{Theorem}[section]
\newtheorem{lem}[thm]{Lemma}
\newtheorem{cor}[thm]{Corollary}
\newtheorem{prop}[thm]{Proposition}
\newtheorem{conj}[thm]{Conjecture}
\newtheorem{question}[thm]{Question}

\theoremstyle{definition}

\newtheorem{example}[thm]{Example}

\theoremstyle{remark}
\newtheorem{rem}[thm]{Remark}

\numberwithin{equation}{section}

\begin{document}

\newcommand{\thmref}[1]{Theorem~\ref{#1}}
\newcommand{\secref}[1]{Section~\ref{#1}}
\newcommand{\lemref}[1]{Lemma~\ref{#1}}
\newcommand{\propref}[1]{Proposition~\ref{#1}}
\newcommand{\corref}[1]{Corollary~\ref{#1}}
\newcommand{\remref}[1]{Remark~\ref{#1}}
\newcommand{\eqnref}[1]{(\ref{#1})}
\newcommand{\exref}[1]{Example~\ref{#1}}

\newcommand{\nc}{\newcommand}
\nc{\Z}{{\mathbb Z}}
\nc{\hZ}{{\underline{\mathbb Z}}}
\nc{\C}{{\mathbb C}}
\nc{\N}{{\mathbb N}}
\nc{\F}{{\mf F}}
\nc{\Q}{\mathbb {Q}}
\nc{\la}{\lambda}
\nc{\ep}{\delta}
\nc{\h}{\mathfrak h}
\nc{\n}{\mf n}
\nc{\A}{{\mf a}}
\nc{\G}{{\mathfrak g}}
\nc{\SG}{\overline{\mathfrak g}}
\nc{\DG}{\widetilde{\mathfrak g}}
\nc{\D}{\mc D}
\nc{\Li}{{\mc L}}
\nc{\La}{\Lambda}
\nc{\is}{{\mathbf i}}
\nc{\V}{\mf V}
\nc{\bi}{\bibitem}
\nc{\NS}{\mf N}
\nc{\dt}{\mathord{\hbox{${\frac{d}{d t}}$}}}
\nc{\E}{\EE}
\nc{\ba}{\tilde{\pa}}
\nc{\half}{\frac{1}{2}}
\nc{\mc}{\mathcal}
\nc{\mf}{\mathfrak} \nc{\hf}{\frac{1}{2}}
\nc{\hgl}{\widehat{\mathfrak{gl}}}
\nc{\gl}{{\mathfrak{gl}}}
\nc{\hz}{\hf+\Z}
\nc{\dinfty}{{\infty\vert\infty}}
\nc{\SLa}{\overline{\Lambda}}
\nc{\SF}{\overline{\mathfrak F}}
\nc{\SP}{\overline{\mathcal P}}
\nc{\U}{\mathfrak u}
\nc{\SU}{\overline{\mathfrak u}}
\nc{\ov}{\overline}
\nc{\wt}{\widetilde}
\nc{\wh}{\widehat}
\nc{\sL}{\ov{\mf{l}}}
\nc{\sP}{\ov{\mf{p}}}
\nc{\osp}{\mf{osp}}
\nc{\spo}{\mf{spo}}
\nc{\hosp}{\widehat{\mf{osp}}}
\nc{\hspo}{\widehat{\mf{spo}}}
\nc{\hh}{\widehat{\mf{h}}}
\nc{\even}{{\bar 0}}
\nc{\odd}{{\bar 1}}
\nc{\mscr}{\mathscr}

\newcommand{\blue}[1]{{\color{blue}#1}}
\newcommand{\red}[1]{{\color{red}#1}}
\newcommand{\green}[1]{{\color{green}#1}}
\newcommand{\white}[1]{{\color{white}#1}}


\newcommand{\aaf}{\mathfrak a}
\newcommand{\bb}{\mathfrak b}
\newcommand{\cc}{\mathfrak c}
\newcommand{\qq}{\mathfrak q}
\newcommand{\qn}{\mathfrak q (n)}
\newcommand{\UU}{\bold U}
\newcommand{\VV}{\mathbb V}
\newcommand{\WW}{\mathbb W}
\nc{\TT}{\mathbb T}
\nc{\EE}{\mathbb E}
\nc{\FF}{\mathbb F}
\nc{\KK}{\mathbb K}

 \advance\headheight by 2pt

\title[Character formulae and canonical bases of type $A/C$]
{Character formulae for queer Lie superalgebras and canonical bases of types $A/C$}

\author[Cheng]{Shun-Jen Cheng$^\dagger$}
\thanks{$^\dagger$Partially supported by a MoST and an Academia Sinica Investigator grant}
\address{Institute of Mathematics, Academia Sinica, Taipei,
Taiwan 10617} \email{chengsj@math.sinica.edu.tw}

\author[Kwon]{Jae-Hoon Kwon$^{\dagger\dagger}$}
\thanks{$^{\dagger\dagger}$Partially supported by Samsung Science and Technology Foundation under Project Number SSTF-BA1501-01.}
\address{Department of Mathematical Sciences, Seoul National University, Seoul 08826, Korea
 }
\email{jaehoonkw@snu.ac.kr}

\author[Wang]{Weiqiang Wang$^{\dagger\dagger\dagger}$}
\thanks{$^{\dagger\dagger}$Partially supported by an NSF grant.}
\address{Department of Mathematics, University of Virginia, Charlottesville, VA 22904}
\email{ww9c@virginia.edu}

\begin{abstract}
For the BGG category of $\mathfrak{q}(n)$-modules of  half-integer weights, a Kazhdan-Lusztig conjecture \`a la Brundan is formulated in terms of categorical canonical basis of the $n$th tensor power of the natural representation of the quantum group of type $C$. For the BGG category of $\mathfrak{q}(n)$-modules of congruent non-integral weights, a Kazhdan-Lusztig conjecture is formulated in terms of canonical basis of a mixed tensor of the natural representation and its dual of the quantum group of type $A$.
We also establish a character formula for the finite-dimensional irreducible $\mathfrak{q}(n)$-modules of half-integer weights in terms of type $C$ canonical basis of the corresponding $q$-wedge space.
\end{abstract}

\subjclass[2010]{17B67}

\maketitle

\setcounter{tocdepth}{1}
\tableofcontents

\section{Introduction}

\subsection{}
The celebrated Kazhdan-Lusztig theory for semisimple Lie algebras is formulated in terms of canonical
bases of the Hecke algebras associated to the corresponding Weyl groups. For basic Lie superalgebras (which can be viewed as super counterparts of semisimple Lie algebras) as well as for the queer Lie superalgebras $\qn$ (see, e.g., \cite{CW}),
triangular decomposition is available, which allows one to define the BGG category $\mathcal O$. However, the linkage in $\mathcal O$ is no longer controlled by
the Weyl groups, and thus the conventional formulation of Kazhdan-Lusztig theory does not apply to Lie superalgebras.

In \cite{CLW} Lam and two of the authors proved  Brundan's Kazhdan-Lusztig conjecture for the BGG category $\mathcal O$ of integer-weight modules for the Lie superalgebra $\mathfrak{gl}(m|n)$  (also see \cite{BLW}); Brundan's conjecture \cite{Br1} is formulated
 in terms of the type $A$ canonical basis of a tensor product module \cite{Lu}.
It is shown in \cite{CMW}  that the irreducible character
problem for the BGG category of $\mf{gl}(m|n)$-modules of arbitrary weights can be reduced to the setting of integer weights by using various category equivalences induced by twisting, odd reflection, and parabolic induction functors.

A Kazhdan-Lusztig theory for the Lie superalgebras $\mathfrak{osp}(2m+1|2n)$
has recently been formulated and established in \cite{BW} via new canonical bases arising from quantum symmetric pairs.
This approach  has been adapted further to establish a Kazhdan-Lusztig theory for $\mathfrak{osp}(2m|2n)$ \cite{B15}.
In these papers, two variants of the category $\mathcal O$ for every $\mathfrak{osp}$-type superalgebra were considered,
one with integer weights and another with half-integer weights, and their corresponding Kazhdan-Lusztig theories utilize  canonical
bases from somewhat different quantum algebras.

\subsection{}

An analogous (still open) conjecture for the characters  of irreducible $\qn$-modules of {\em integer weights} in the category $\mathcal O$ was formulated in \cite{Br2} in terms of
a {\em type $B$} canonical basis \cite{Lu90, Lu, Ka91}; see Remark~\ref{rem:remark on conj for integral weights} for an update and a modification.
Brundan established a character formula for the
finite-dimensional irreducible modules of integer weights in terms of type $B$ dual canonical basis of a $q$-deformed wedge space.
However there is no reason to restrict oneself to the integer-weight modules only.
Characters for the finite-dimensional irreducible $\qn$-modules of half-integer weights have recently been obained in \cite{CK},
and they were shown to be closely related to those of the finite-dimensional irreducible
characters of integer weights given in \cite{Br2}. We remark here that an algorithm for all finite-dimensional irreducible characters of $\qn$-modules was developed earlier in  \cite{PS2}.

We are interested in understanding the characters of irreducible modules in the whole BGG category of $\qn$-modules, but not just those with integer-weight highest weights.
A reduction similar to (but more complicated than) \cite{CMW} is established for a queer Lie superalgebra by Chih-Whi Chen in \cite{Chen}.
By Chen's result, which confirmed our expectation,
the problem of computing the characters of the irreducible modules of arbitrary highest weights in the category $\mc O$ for a queer Lie superalgebra is reduced to the problem of
computing them in the following three categories:
(i)~ a BGG category $\mc O_{n,\Z}$ of the $\qn$-modules of integer weights (see the main conjecture of \cite[(4.56)]{Br2}),
(ii)~ a BGG category $\mc O_{n,\hZ}$ (where $\hZ =\hf +\Z$) of the $\qn$-modules of half-integer weights,
(iii)~ a BGG category $\mc O_{n,s^l}$ of the $\qn$-modules of ``congruent $\pm s$-weights",
for $s \in \C \setminus \hf \Z$ and $l \in \{0, 1, \ldots,n\}$ (see Section~ \ref{subsec:KL:A} for a precise definition).

\subsection{}

A main goal of this paper is to formulate a Kazhdan-Lusztig conjecture for the categories $\mc O_{n,\hZ}$  (Conjecture~ \ref{conj:half-integer2})
and $\mc O_{n,s^l}$ (Conjecture~ \ref{conj:KL conj for s-weights}).
We also prove results {and formulate various conjectures} on connections among the canonical bases of types $A/B/C$, which seem to indicate new connections
between representations of $\qn$ and general linear Lie algebra $\mathfrak{gl}(l|n-l)$ for various $l$. {The present paper contains several conjectures, and the reader is invited to prove some (or all) of them. It is also our hope that some of the heuristic viewpoints in the paper may inspire the reader to formulate his/her own conjectures!}

\subsection{}
The Kazhdan-Lusztig conjecture for $\mc O_{n,\hZ}$ in the first version of this paper (see Conjecture~ \ref{conj:half-integer})
is formulated in terms of canonical basis of the tensor module $\VV^{\otimes n}$,
where $\VV$ is the natural representation of the quantum group $\UU=\UU_q(\mf c_\infty)$ of {\em type $C$}.
However, Tsuchioka's computation \cite{Tsu} shows this canonical basis (as well as analogous canonical basis of type $B$ in \cite{Br2}; see \eqref{TLMN})
have no positivity property, and hence our ``type $C$" Conjecture~ \ref{conj:half-integer} as well as Brundan's original ``type $B$" KL conjecture require correction.
But there is plenty of evidence (e.g. translation functor, linkage) supporting the relevance of type $C/B$ canonical bases in the representation theory of the queer Lie superalgebra.
To this end, the translation functors on the level of Grothendieck group of $\mc O_{n,\hZ}$ are shown to satisfy the type $C$ Lie algebra relations.

The situation is reminiscent of the recent advances in algebraic groups, where canonical bases of affine Hecke algebras are replaced by $p$-canonical bases; see Riche-Williamson \cite{RW}. Recall that Webster \cite{Web} has categorified the tensor product of highest weight integrable modules, and so there is a categorical canonical basis
(called an orthodox basis by Webster, or  can$\oplus$nical basis here to indicate positivity)
available in place of canonical basis of $\VV^{\otimes n}$.
We reformulate the Kazhdan-Lusztig conjecture for $\mc O_{n,\hZ}$
in terms of Webster's  can$\oplus$nical basis of $\VV^{\otimes n}$  (see Conjecture~ \ref{conj:half-integer2}).

Moreover, we formulate an ``$C=A$" conjecture that the tilting characters in a full subcategory ${}_{\min} \mc O^\Delta_{n,\hZ}$ of $\mc O^\Delta_{n,\hZ}$
(which correspond  to the subspace \eqref{TTs:min})
are solved by the  type $A$ canonical bases; see Conjecture~\ref{conj:KLC=A}.
As a supporting evidence we show that the type $C$ canonical basis in the subspace \eqref{TTs:min} are indeed type $A$ canonical basis; see Corollary~\ref{cor:ACB=CCB}.
For this subspace, we conjecture that Webster's can$\oplus$nical basis coincides with the Lusztig's type $C$ (and hence type $A$) canonical basis.

We also establish  (see Theorem \ref{thm:Main}) a character formula for the finite-dimensional irreducible modules in $\mc O_{n,\hZ}$ in terms of type $C$
dual canonical basis of the corresponding $q$-wedge space.

Conjecture~ \ref{conj:KL conj for s-weights} for $\mc O_{n,s^l}$ is formulated in terms of {\em type $A$} canonical basis on
$\VV^{\otimes l}_s \otimes {\VV^*_s}^{\otimes n-l}$, where $\VV_s$ is the natural representation of $\UU_q(\mf{sl}_\infty)$ (just as in \cite{Br1, CLW} for $\mf{gl}(l|n-l)$); in particular, the conjecture of $\mc O_{n,s^l}$ is independent of the scalars $s \not \in \hf \Z$.

While our formulation is analogous to \cite{Br2},  we find it remarkable to see a natural
appearance of canonical/can$\oplus$nical basis arising from yet
another classical quantum group in the Kazhdan-Lusztig formulation for Lie superalgebras. As far as we know, there is no connection between
the BGG category of a semisimple Lie algebra and
the canonical/can$\oplus$nical bases of types $B/C/D$.

We prove that the type $C$ canonical basis on a $q$-wedge space of the natural representation indeed coincides with that of type $A$ canonical basis.

\subsection{}

The paper is organized as follows.
In Section~\ref{sec:prelim}, we recall the definition of the quantum group $\UU$, its tensor module $\VV^{\otimes n}$ and the corresponding $q$-deformed wedge module.
We also introduce a certain completion of $\VV^{\otimes n}$ in which the canonical basis live.

In Section~\ref{sec:CB}, we construct the canonical basis on a tensor module, which is a natural infinite-rank extension
of Lusztig's construction \cite{Lu}, and then the canonical basis on the $q$-wedge module.
This $q$-wedge space arises as a limiting case of the constructions in \cite{JMO}. This is done similarly as in \cite{Br2} (see also \cite{CLW}).

In Section~\ref{sec:CBabc}, we compare the type $C$ canonical basis on $\VV^{\otimes n}$ and the type $A$ canonical basis on the same space
(viewed as a module over the type $A$ subalgebra $\UU^{\mf a}$ of $\UU$). We show that
these two types of canonical bases coincide on one particular subspace of $\VV^{\otimes n}$, which plays a crucial role in the next section.

The representation theory of $\qn$ is studied in Section~\ref{sec:O}. We establish a character formula for the finite-dimensional irreducible $\qn$-modules of half-integer weights.
We then formulate in Conjecture~ \ref{conj:half-integer2} a
character formula for $\mc O_{n,\hZ}$ in terms of Webster's type $C$ can$\oplus$nical basis on $\VV^{\otimes n}$. Finally, in Conjecture \ref{conj:KL conj for s-weights} we formulate a character formula for $\mc O_{n,s^l}$ in terms of type $A$ (Lusztig's) canonical basis on $\VV^{\otimes l}_s \otimes {\VV^*_s}^{\otimes n-l}$.

In Appendix \ref{sec:notations for B}, we revisit the setting of \cite{Br2} for the category $\mc O_{n,\Z}$,
and suggest a modification of Brundan's original conjecture in Remark~\ref{rem:remark on conj for integral weights}.
We establish in Proposition \ref{thm:ACB=BCB} an equality of certain
canonical bases of types $B$ and $A$.\vskip 3mm

\vspace{.2cm}  \noindent {\bf Acknowledgements.}
We thank Shunsuke Tsuchioka for showing us his computation on type $B/C$ canonical bases,
which has led us to revise our ``type $C$" KL conjecture as well as Brundan's ``type $B$" KL conjecture.
The last two authors thank Institute of Mathematics, Academia Sinica, Taipei, for hospitality and support. We also would like to thank the referees for careful reading and helpful comments.

\vspace{.2cm} \noindent {\bf Notation.} We denote by $\Z$, $\N$, and $\Z_+$  the sets of all, positive, and
non-negative integers, respectively.
We also denote by $\hZ=\hf+\Z$ and $\hZ_+$  the set of half-odd-integers and positive half-odd-integers, respectively.

\vspace{.2cm} \noindent {\bf Notes added.} A proof of a maximal parabolic version of Conjecture~\ref{conj:KL conj for s-weights} has appeared in \cite{CC}. The proofs of
Conjecture \ref{conj:half-integer2} and (the full version) Conjecture~\ref{conj:KL conj for s-weights} will appear in a forthcoming work by Brundan and Davidson \cite{BN} using the unicity of categorification.
Indeed they noticed independently the corrected formulation in Conjecture \ref{conj:half-integer2} (upon knowing Conjecture~ \ref{conj:half-integer} and \cite{Tsu})
and Brundan's original conjecture.

\section{The tensor and wedge modules of quantum groups of type $C$}
\label{sec:prelim}

\subsection{The quantum group $\UU_q(\mf c_\infty)$}
Let $P$ be the free abelian group on basis $\delta_r$ ($r\in\hZ_+$) equipped  with a symmetric bilinear form $(\ep_r,\ep_s)=\delta_{rs}$, for $r, s\in \hZ_{+}$.
For later use we set $\delta_{-r} =-\delta_r$, for $r \in \hZ_+$.
Let $\mf c_\infty$ be the Kac-Moody Lie algebra of type $C$ of infinite rank whose associated Dynkin diagram and simple roots are given by

\begin{center}
\hskip -1cm  \setlength{\unitlength}{0.16in}
\begin{picture}(24,4)

\put(5.6,2){\makebox(0,0)[c]{$\bigcirc$}}
\put(8,2){\makebox(0,0)[c]{$\bigcirc$}}
\put(10.4,2){\makebox(0,0)[c]{$\bigcirc$}}
\put(14.85,2){\makebox(0,0)[c]{$\bigcirc$}}
\put(17.25,2){\makebox(0,0)[c]{$\bigcirc$}}
\put(19.4,2){\makebox(0,0)[c]{$\bigcirc$}}
\put(8.35,2){\line(1,0){1.5}} \put(10.82,2){\line(1,0){0.8}}
\put(13.2,2){\line(1,0){1.2}} \put(15.28,2){\line(1,0){1.45}}
\put(17.7,2){\line(1,0){1.25}} \put(19.81,2){\line(1,0){0.9}}
\put(6.8,1.97){\makebox(0,0)[c]{$\Longrightarrow$}}
\put(12.5,1.95){\makebox(0,0)[c]{$\cdots$}}
\put(21.5,1.95){\makebox(0,0)[c]{$\cdots$}}
\put(5.6,1){\makebox(0,0)[c]{\tiny ${\alpha}_0$}}
\put(8,1){\makebox(0,0)[c]{\tiny ${\alpha}_1$}}
\put(10.4,1){\makebox(0,0)[c]{\tiny ${\alpha}_2$}}
\put(15,1){\makebox(0,0)[c]{\tiny ${\alpha}_{k-1}$}}
\put(17.4,1){\makebox(0,0)[c]{\tiny ${\alpha}_k$}}
\put(19.8,1){\makebox(0,0)[c]{\tiny ${\alpha}_{k+1}$}}
\end{picture}
\end{center} \vskip -5mm
\begin{equation*}
\begin{split}
& \alpha_0=-2\delta_{\hf} , \ \alpha_i=\delta_{i-\hf}-\delta_{i+\hf} \quad (i\geq 1).
\end{split}
\end{equation*}
Let $\left(a_{ij}\right)$ be the Cartan matrix of $\mf c_\infty$, where $a_{ij}=2(\alpha_i,\alpha_j)/(\alpha_i,\alpha_i)$, for $i,j\geq 0$.
Let $Q$ be the root lattice of $\mf c_\infty$ and $Q^+=\sum_{i\geq 0}\Z_+\alpha_i$. We define the usual dominance ordering $\ge$ on $P$ with respect to $\mf c_\infty$:
\begin{equation}
\label{eq:geq}
\text{We  let $f\ge g$ if $f-g\in Q^+$, for $f,g\in P$.}
\end{equation}

Let $q$ be an indeterminate. We put $q_i=q^{\frac{(\alpha_i,\alpha_i)}{2}}$, $[m]_i=\frac{q_i^m-q_i^{-m}}{q_i-q_i^{-1}}$, and $[m]_i!=[1]_i[2]_i\cdots [m]_i$, for $i\geq 0$ and $m\in\N$.
Let $\UU=\UU_q(\mf c_\infty)$ denote the quantum group associated to $\mf c_\infty$, which is an associative algebra with $1$ over $\mathbb Q(q)$ with generators $K_i,E_i,F_i$ ($i\ge 0$) subject to the relations:
\begin{equation}\label{eq:defining relations for QG}
\begin{split}
&K_iK_i^{-1}=1,\quad K_iK_j=K_jK_i,\\
&K_iE_jK_i^{-1}=q^{(\alpha_i,\alpha_j)}E_j,\quad K_iF_jK_i^{-1}=q^{-(\alpha_i,\alpha_j)}F_j,\\
&E_iF_j-F_jE_i=\delta_{ij}\frac{K_i-K_i^{-1}}{q_i-q_i^{-1}},\\
&\sum_{m=0}^{1-a_{ij}}(-1)^m E_i^{(m)}E_j E_i^{(1-a_{ij}-m)}=\sum_{m=0}^{1-a_{ij}}F_i^{(m)}F_j F_i^{(1-a_{ij}-m)}=0\quad (i\neq j),
\end{split}
\end{equation}
for $i,j\geq 0$, where $E_i^{(m)}=E_i^m/[m]_i!$ and $F_i^{(m)}=F_i^m/[m]_i!$.

Let $\UU^\pm$ be the subalgebra of $\UU$ generated by $E_i$ $(i\geq 0)$ and $F_i$ $(i\geq 0)$, respectively, and let $\UU^0$ be the subalgebra generated by $K_i$  $(i\geq 0)$.
For $\nu\in Q^+$, let $\UU^\pm_{\nu}$ be the subspace of $\UU^\pm$ spanned by $E_{i_1}\cdots E_{i_r}$ (respectively, $F_{i_1}\cdots F_{i_r}$) such that $\alpha_{i_1}+\ldots+\alpha_{i_r}=\nu$. Then $\UU^\pm=\bigoplus_{\nu\in Q^+}\UU^\pm_{\nu}$, and $\UU \cong \UU^+\otimes \UU^0\otimes \UU^-$ as a $\Q(q)$-space.
Let $\UU_{\Z[q^{\pm}]}$ be the $\Z[q,q^{-1}]$-subalgebra of $\UU$ generated by $E_i^{(r)}$, $F_i^{(r)}$, $K_i^{\pm1}$, and
\begin{equation*}
\left[ \begin{array}{c} K_i \\ r \end{array} \right]:= \prod_{s=1}^r\frac{K_iq_i^{1-s} - K^{-1}_iq_i^{s-1}}{q_i^s-q_i^{-s}},
\end{equation*}
for $i\geq 0$ and $r\geq 1$.
Recall that
$\UU$ is a Hopf algebra with a comultiplication $\Delta$ (different from the one used in \cite{Lu}) given by
\begin{equation}\label{comul}
\begin{split}
&\Delta(E_i)=1\otimes E_i+E_i\otimes K_i^{-1},\\
&\Delta(F_i)=K_i\otimes F_i+F_i\otimes 1,\\
&\Delta(K_i)=K_i\otimes K_i,
\end{split}
\end{equation}
for $i\ge 0$.
Let $\overline{\phantom{x}}$ be the $\Q(q)$-antilinear involution of $\UU$ given by $\ov{q}=q^{-1}$,
$\ov{E_i}=E_i$, $\ov{F_i}=F_i$, and $\ov{K_i}=K^{-1}_i$ ($i\geq 0$), and let $\sigma$ be the $\Q(q)$-linear anti-involution of $\UU$ given by $\sigma(E_i)=E_i$, $\sigma(F_i)=F_i$, and $\sigma(K_i)=K_i^{-1}$ ($i\geq 0$). 

\subsection{Bruhat orderings}
\label{subsec:bruhat ordering}

Recall $\hZ =\hf+\Z$.
We use the following notations for various sets of weights (suppressing the dependence on $n>0$):
\begin{equation}\label{eq:Lambda over half integer}
\begin{split}
&\La_{\hZ}:=\{\,\la=(\la_1,\ldots,\la_n)\,|\,\la_i\in\hZ,\  \forall i \,\},\\
&\La^{\geqslant}_{\hZ}:=\{\,\la \in \La_{\hZ}\,|\,\la_1\ge\la_2\ge\cdots\ge\la_n\,\},\\
&\La^{+}_{\hZ}:=\{\,\la \in\La_{\hZ}\,|\,\la_1>\la_2>\cdots>\la_n\,\}.
\end{split}
\end{equation}

For $\la=(\la_1,\ldots,\la_n)\in \Lambda_{\hZ}$ and $1\leq r\leq n$, we let
\begin{align}\label{eq:wt}
{\rm wt}_r(\la):=\sum_{i=r}^n\delta_{\la_i} \in P, \quad {\rm wt}(\la):={\rm wt}_1(\la)\in P.
\end{align}
Recall from  \eqref{eq:geq} the partial ordering $\ge$ on the lattice $P$.
For $\la,\mu\in\La_{\hZ}$, we define $\la\succeq\mu$ if ${\rm wt}_r(\la)\ge{\rm wt}_r(\mu)$, for $1\leq r\leq n$ and ${\rm wt}_1(\la)={\rm wt}_1(\mu)$. We call $\succeq$ the {\em Bruhat ordering on $\La_{\hZ}$}. Given $\la, \mu\in\Lambda_\hZ$ with $\la\succeq \mu$, there are only finitely many $\nu\in \Lambda_\hZ$ such that $\la\succeq\nu\succeq\mu$, which can be seen by using an analogue of \cite[Lemma 2.15]{Br2} for $\Lambda_\hZ$.

Let $\{\,\varepsilon_i\,|\,1\leq i\leq n\,\}$ be the standard orthonormal basis of $\Q^n$, and we regard $\La^{\geqslant}_{\hZ} \subset \Q^n$. For $\la, \mu\in \La^\geqslant_{\hZ}$, we define another { Bruhat ordering} $\succcurlyeq$ as in \cite[Lemma 2.1]{PS2}: $\la\succcurlyeq \mu$ if and only if there exists a sequence of elements $\mu=\nu_{(1)},\ldots,\nu_{(k)}=\la$ in $\La^{\geqslant}_{\hZ}$ and roots $\varepsilon_{s_i}-\varepsilon_{t_i}$ with $s_i<t_i$ such that $\nu_{(i)}+(\varepsilon_{s_i}-\varepsilon_{t_i})=\nu_{(i+1)}$ and $(\nu_{(i)},\varepsilon_{s_i}+\varepsilon_{t_i})=0$ for $1\leq i\leq k-1$. Then by the same argument as in \cite[Proposition 3.3]{CK}, we can prove the following:

\begin{prop}\label{prop:Bruhat PS=Br}
For $\la, \mu\in \La^{\geqslant}_{\hZ}$, we have $\la\succeq \mu$ if and only if $\la\succcurlyeq \mu$.
\end{prop}


\subsection{The $q$-tensor and $q$-wedge spaces}
 \label{subsec:completion}

Let
\begin{equation}
 \label{eq:V}
\VV :=\bigoplus_{r\in\hZ}\Q(q) v_r
\end{equation}
be the natural representation of $\UU$, where the action of $\UU$ is given by
\begin{equation*}\label{eq:natural_rep_C}
\begin{split}
&E_0v_r=\delta_{r,\hf}v_{-\hf},\quad \quad\quad\quad\ \ E_iv_r=\delta_{r,i+\hf}v_{i-\hf}+\delta_{r,-i+\hf}v_{-i-\hf},\\
&F_0v_r=\delta_{r,-\hf}v_\hf,\quad \quad\quad\quad\ \ F_iv_r=\delta_{r,i-\hf}v_{i+\hf}+\delta_{r,-i-\hf}v_{-i+\hf},\\
&K_0v_r=q^{-2\delta_{r,\hf}+2\delta_{r,-\hf}}v_r,\quad K_iv_r=q^{\delta_{r,i-\hf}+\delta_{r,-i-\hf}-\delta_{r,-i+\hf}-\delta_{r,i+\hf}}v_r,
\end{split}
\end{equation*}
for $i\geq 1$ and $r\in \hZ$ (cf.~\cite[Chapter 5A]{Jan}).
We put
\begin{equation}
\label{eq:TT}
\TT:=\bigoplus_{n\geq 0}\TT^n, \quad \text{where}\quad \TT^n:=\VV^{\otimes n},
\end{equation}
which is a $\UU$-module  via \eqref{comul}.
Let $\KK$ be the two-sided (homogeneous) ideal of $\TT$ generated by
\begin{equation}\label{eq:wedge_rel_C}
\begin{split}&v_r\otimes v_r,\quad v_r\otimes v_s +q v_s\otimes v_r \quad (r>s,\ r+s\not=0),\\
&v_r\otimes v_{-r}+q(v_{r-1}\otimes v_{1-r}+v_{1-r}\otimes v_{r-1})+q^2v_{-r}\otimes v_r \quad (r>1/2),\\
&v_\hf\otimes v_{-\hf}+q^2v_{-\hf}\otimes v_\hf.
\end{split}
\end{equation}
We have $\KK=\bigoplus_{n\geq 0}\KK^n$, where $\KK^n =\KK \cap \TT^n$.
The quotient $\FF=\TT / \KK$ is called the $q$-wedge space of type $C$, and
it is essentially a limit case of \cite[Proposition 2.3]{JMO}. Let $\pi : \TT \rightarrow \FF$ be the canonical projection. We have
\begin{equation*}
\FF=\bigoplus_{n\geq 0}\FF^n,\quad \text{where}\quad \FF^n := \TT^n/\KK^n.
\end{equation*}
One sees that $\KK^n$ is a $\UU$-submodule of $\TT^n$, and hence $\FF^n$ is a $\UU$-module.

Fix $n\geq 1$.  We set
\begin{align*}
M_\la &=v_{\la_1}\otimes \cdots \otimes v_{\la_n}, \quad \text{ for } \la\in \Lambda_\hZ,
\\
F_\la &=\pi(M_{w_0\la}), \quad\quad\quad\ \ \text{ for } \la\in \Lambda_\hZ^+,
\end{align*}
where $w_0$ is the longest element in the symmetric group $\mf S_n$ acting on $\La_{\hZ}$ as permutations. The set $\{\,M_\la\,|\,\la\in \Lambda_\hZ\,\}$ forms a $\Q(q)$-basis of $\TT^n$. We define a symmetric bilinear form $(\ ,\ )$ on $\TT^n$ by
\begin{equation}\label{eq:bilinear form on Tn}
(M_\la,M_\mu)=\delta_{\la \mu},
\end{equation}
for $\la,\mu\in \Lambda_\hZ$.
In general, if $\la\in\La_\hZ\setminus \La^+_\hZ$, then $\pi(M_{w_0\la})$ is a $q\Z[q]$-linear combination of $F_\mu$'s with $\mu\succ \la$ by \eqref{eq:wedge_rel_C} (cf.~\cite[Lemma 3.2]{Br2}). Also note that we have $\mu\succ \la$ if and only if $w_0\mu\prec w_0\la$ for $\la, \mu\in \La_\hZ$. The set $\{\,F_\la\,|\,\la\in \La^{+}_{\hZ}\,\}$ forms a $\Q(q)$-basis of $\FF^n$, which is a direct consequence of \cite[Theorem 1.2]{Bg}.

For $k\in \N$, let $\TT^n_{< k}$ be the subspace of $\TT^n$ spanned by $M_\la$, for $\la=(\la_1,\ldots,\la_n)\in\La_{\hZ}$ with $|\la_i|< k$ for $1\leq i\leq n$. We take the completion $\widetilde{\TT}^n$ of $\TT^n$ with respect to the descending filtration $\{{\rm Ker}\pi_k \}_{k\in\N}$, where
\begin{equation}\label{eq:projection pi_k}
\pi_k: \TT^n \rightarrow \TT^n_{< k}
\end{equation}
is the canonical projection.
The set $\{\,M_\la\,|\,\la\in \La_{\hZ}\,\}$ forms a topological $\Q(q)$-basis of $\widetilde{\TT}^n$.
Then we define $\widehat{\TT}^n$ to be the subspace of $\widetilde{\TT}^n$ spanned by vectors of the form
$M_\la+\sum_{\mu\prec \la}c_\mu M_\mu,$
for $\la\in\La_{\hZ}$ with $c_\mu\in \Q(q)$, which is possibly an infinite sum.  

Let $\widetilde{\KK}^n$ be the completion of $\KK^n$ in $\widetilde{\TT}^n$, and $\widehat{\KK}^n=\widetilde{\KK}^n\cap \widehat{\TT}^n$
the closure of $\KK^n$ in $\widehat{\TT}^n$. Then we define
\begin{equation*}\label{eq:completion of Fn}
\widehat{\FF}^n=\widehat{\TT}^n/\widehat{\KK}^n.
\end{equation*}
We see that $\widehat{\FF}^n$ is spanned by the vectors of the form
$F_\la+\sum_{\mu\succ\la}c_\mu F_\mu,$
for $\la\in\La^+_\hZ$ with $c_\mu\in \Q(q)$, which is also possibly an infinite sum (cf.~\cite[Lemma~ 3.2]{Br2}). 

\begin{rem}{\rm
We can consider the completions of $\TT^n$ and $\FF^n$ as in \cite{Br2} with respect to the grading on $\La_\hZ$ given by $\sum_{d=1}^n d\la_d$, which indeed properly contain $\widehat{\TT}^n$ and $\widehat{\FF}^n$, respectively (see Remark \ref{rem:completions}).
We remark that all the results on (dual) canonical bases, which will be given in later sections, remain unchanged if we choose to work with these completions instead.
}
\end{rem}

\section{Canonical basis on the tensor module and wedge module}
\label{sec:CB}

\subsection{Quasi-$\mc R$-matrix}
\label{subsec:PBW and quasi-R}

For $k\geq 1$, let $\mf c_k$ be the subalgebra of $\mf c_\infty$ corresponding to $\{\,\alpha_i\,|\,0\leq i\leq k-1\,\}$ with positive root lattice $Q^+_k=\sum_{0\leq i\leq k-1}\Z_+\alpha_i$.
We regard the quantum group $\UU_q(\mf{c}_k)$ as the subalgebra generated by $K_i$, $E_i$, and $F_i$, for $0\leq i\leq k-1$.

Put $p=\frac{k(k+1)}{2}$ and $N=k^2$.
Let ${\bf h}=(i_1,i_2,\ldots,i_{N})$ be the sequence given by
\begin{equation*}
\begin{split}
&(i_1,\ldots,i_{p})=(0,1,\ldots,k-1,0,1,\ldots,k-2,\ldots,0,1,0), \\
&(i_{p+1},\ldots,i_{N})=(1,2,1,3,2,1,\ldots,k-1,\ldots,1).
\end{split}
\end{equation*}
Note that $w^{(k)}_0=s_{i_1}\ldots s_{i_N}$
is the longest element in the Weyl group of $\mf{c}_k$, which is generated by the simple reflections $s_i$, for $0\leq i\leq k-1$.
For ${\bf c}=(c_1,\ldots,c_N)\in\Z_+^N$, consider the vectors of the following form:
\begin{equation}\label{eq:PBW vector}
\begin{split}
E^{\bf c}_{{\bf h},p}=&
E^{(c_{p})}_{i_{p}}T'_{i_{p},-1}(E^{(c_{p-1})}_{i_{p-1}})\cdots T'_{i_{p},-1}T'_{i_{p-1},-1}\cdots T'_{i_{2},-1}(E^{(c_{1})}_{i_{1}})
\\
& \times
T''_{i_{p+1},1}T''_{i_{p+2},1}\cdots T''_{i_{N-1},1}(E^{(c_{N})}_{i_{N}}) \cdots T''_{i_{p+1},1}(E^{(c_{p+2})}_{i_{p+2}})E^{(c_{p+1})}_{i_{p+1}},
\end{split}
\end{equation}
(see \cite[Section 38.2.2]{Lu}), which has been adjusted according to our different convention of comultiplication \eqref{comul}.
 Here $T'_{i,e}$, $T''_{i,e}$ ($e=\pm 1$) are the automorphisms of $\UU$ given in \cite[Section 37.1]{Lu}.

By \cite[Proposition 38.2.3]{Lu}, the set ${\bf B}^{(k)}:=\{\,E^{\bf c}_{{\bf h},p}\,|\,{\bf c}\in\Z_+^N\,\}$ is a $\Q(q)$-basis of $\UU^+\cap\UU_q(\mf c_k)$
and ${\bf B}^{(k)}_\nu:={\bf B}^{(k)}\cap \UU^+_\nu$   is a $\Q(q)$-basis of $\UU^+_\nu\cap\UU_q(\mf c_k)$ for $\nu\in Q^+_k$.
Since
\begin{equation*}
w^{(k+1)}_0=w'w^{(k)}_0w'',
\end{equation*}
with $w'=s_0s_1\ldots s_{k}$ and $w''=s_{k}\ldots s_1$, we have ${\bf B}^{(k)}\subset {\bf B}^{(k+1)}$. Therefore ${\bf B}=\bigcup_{k\geq 2}{\bf B}^{(k)}$ is a $\Q(q)$-basis of $\UU^+$, and ${\bf B}=\bigsqcup_{\nu\in Q^+}{\bf B}_\nu$, where ${\bf B}_\nu:={\bf B}\cap \UU^+_\nu$. Note that ${\bf B}_\nu={\bf B}^{(k)}_\nu$ if $\nu\in Q^+_k$ for some $k\in\N$.

Let
\begin{equation}\label{eq:Theta^k}
\Theta^{(k)}:=\sum_{\nu\in Q^+_k}\Theta_\nu
\end{equation}
be the quasi-$\mc R$-matrix for $\UU_q(\mf c_k)$, which is defined in a certain completion of $\UU_q(\mf c_k)^+ \otimes \UU_q(\mf c_k)^-$ \cite[Section 4.1]{Lu} with $\UU_q(\mf c_k)^\pm=\UU^\pm\cap\UU_q(\mf c_k)$.
For $\nu=\sum_i\nu_i\alpha_i\in Q^+_k$, the $\nu$-component $\Theta_\nu$ can be given by
\begin{equation}\label{eq:Theta_nu}
\Theta_\nu=\prod_{i=0}^{k-1} (-q_i)^{-\nu_i}
\sum_{E^{{\bf c}}_{{\bf h},p}\in {\bf B}^{(k)}_\nu}
\prod_{s=1}^N\prod_{t=1}^{c_s}(1-q_{i_s}^{2t})
E^{{\bf c}}_{{\bf h},p}\otimes F^{{\bf c}}_{{\bf h},p},
\end{equation}
where $F^{{\bf c}}_{{\bf h},p}$ is obtained by replacing $E_i$'s with $F_i$'s in the expansion of $E^{{\bf c}}_{{\bf h},p}$ \cite[Theorem 4.1.2]{Lu}.

Let $\mf{a}_{\infty}$ be the subalgebra of $\mf c_\infty$ corresponding to $\{\,\alpha_i\,|\,i\geq 1\,\}$ with positive root lattice $Q_{\mf a}^+=\sum_{i\geq 1}\Z_+\alpha_i$, and let $\UU^{\mf a}:=\UU_q(\mf{a}_{\infty})$ be the associated quantum group. We shall regard $\UU^{\mf a}$ as the $\Q(q)$-subalgebra of $\UU$ generated by $E_i,F_i$ for $i\ge 1$.
The quasi-$\mc R$-matrix for $\UU^{\mf a}\cap \UU_q(\mf c_k)= \UU_q(\mf{sl}(k))$ is given by
\begin{equation}\label{eq:Theta_J}
\Theta_{\mf a}^{(k)}:= \sum_{\nu\in Q^+_{\mf a}\cap Q^+_k}\Theta_\nu,
\end{equation}
where $\Theta_\nu$ is given in \eqref{eq:Theta_nu}; in this case $\Theta_\nu$ is the sum over
$E^{\bf c}_{{\bf h},p}\in {\bf B}^{(k)}_\nu$ with $c_i=0$ for $i\le p$ since $w^{(k)}_{\mf a}=s_{i_{p+1}}\ldots s_{i_n}$ is the longest element in the Weyl group of $\mf{sl}(k)$.

\subsection{Canonical basis of $\widehat{\TT}^n$}\label{subsec:canonical basis of Tn}

We shall introduce canonical and dual canonical bases of $\widehat{\TT}^n$ and formulate their basic properties
 similar to the ones in the case of type $B$ \cite[Sections 2.4 and 2.5]{Br2}  (see also Appendix \ref{sec:notations for B}).

For $k\geq 1$, let $\psi^{(k)}$ denote the involution on the $\UU_q(\mf c_k)$-module $\TT^{n}_{< k}$ defined with respect to the quasi-$\mc R$-matrix $\Theta^{(k)}$ in \eqref{eq:Theta^k} (see \cite[Section 27.3.1]{Lu}). Let $\la\in \La_\hZ$ be given such that $|\la_i|<k$, for $1\leq i\leq n$. One can show as in \cite{CLW} that $\psi^{(k)}(M_\la)=\pi_k(\psi^{(\ell)}(M_\la))$ for all $\ell\ge k$, where $\pi_k$ is as in \eqref{eq:projection pi_k}, and hence there exists a well-defined limit $\psi(M_\la)=\lim_{k\rightarrow \infty}\psi^{(k)}(M_\la)$ in $\widehat{\TT}^n$ such that
\begin{equation}\label{eq:bar on M}
\psi(M_\la)\in M_\la + \sum_{\mu\prec \la} \Z[q,q^{-1}]M_\mu,
\end{equation}
which is possibly an infinite sum. Furthermore, $\psi$ extends to a bar involution,
denoted by $\overline{\phantom{x}}$,  on $\widehat{\TT}^n$, which is compatible with the bar involution $\overline{\phantom{x}}$ of $\UU$.

By \cite[Lemma 24.2.1]{Lu} (cf.~\cite[Lemma 3.8]{CLW}), there exist unique bar-invariant topological $\Q(q)$-bases $\{\,T_\la\,|\,\la\in\La_{\hZ}\,\}$ and $\{\,L_\la\,|\,\la\in\La_{\hZ}\,\}$ of $\widehat{\TT}^n$ such that
\begin{equation}\label{can:dual:can}
\begin{split}
T_\la=M_\la+\sum_{\mu\prec\la}t_{\mu\la}(q)M_\mu,\quad
L_\la=M_\la+\sum_{\mu\prec\la}\ell_{\mu\la}(q)M_\mu,
\end{split}
\end{equation}
where $t_{\mu\la}(q)\in q\Z[q]$ and $\ell_{\mu\la}(q)\in q^{-1}\Z[q^{-1}]$. We
set $t_{\la\la}(q)=\ell_{\la\la}(q)=1$, and $t_{\mu\la}(q)=\ell_{\mu\la}(q)=0$ when $\mu\not \preceq\la$.
We call $\{\,T_\la\,|\,\la\in\La_{\hZ}\,\}$ and $\{\,L_\la\,|\,\la\in\La_{\hZ}\,\}$ the {\em canonical} and {\em dual canonical basis} of $\widehat{\TT}^n$, respectively.

Let $\sigma$ be the $\Q(q)$-antilinear automorphism of $\TT^n$ given by $\sigma(M_\la)=M_{-\la}$, for $\la\in \La_{\hZ}$.
We define a bilinear form $\langle \cdot, \cdot \rangle$ on $\widehat{\TT}^n$ by
\begin{align}\label{bi:form}
\langle u,v\rangle:=\left(u,\sigma(\ov{v})\right), \quad \text{ for } u, v\in \widehat{\TT}^n.
\end{align}
Note that although $\sigma(\ov{v})\not\in\widehat{\TT}^n$, the form \eqref{bi:form} is nevertheless well-defined. Also \eqref{bi:form} is symmetric (cf.~\cite[Lemma 2.21]{Br2}), and furthermore,
similar to \cite[Lemma~ 2.25]{Br2}, $\{\,T_\la\,|\,\la\in\La_{\hZ}\,\}$ and $\{\,L_\la\,|\,\la\in\La_{\hZ}\,\}$ are dual in the sense that
\begin{equation}\label{eq:TL pairing}
\langle T_\la, L_{-\mu} \rangle =\delta_{\la\mu},\quad \text{ for } \la, \mu \in \La_{\hZ}.
\end{equation}

\begin{example}\label{ex:CB:n=2}
Let us compute the canonical basis elements $T_{(r,s)}$ ($r,s\in \hZ$) in $\widehat{\TT}^2$.
By \eqref{can:dual:can}, we first have
\begin{equation*}
T_{(r,s)}=M_{(r,s)}\quad (r\le s,r+s\not=0).
\end{equation*}
The other canonical basis elements $T_{(r,s)}$  can be computed by applying appropriate Chevalley generators $E_i$'s and $F_i$'s to $T_{(r,s)}=M_{(r,s)}$ with $r\le s$ and $r+s\neq 0$, and using the characterization of canonical basis elements:
\begin{equation*}
\begin{split}
T_{(r,s)}&= M_{(r,s)}+qM_{(s,r)} \quad (r>s, r+s\not=0),\\
T_{(-r,r)}&= M_{(-r,r)}+qM_{(-r-1,r+1)} \quad (r\ge 1/2),\\
T_{(\hf,-\hf)}&= M_{(\hf,-\hf)}+q^2M_{(-\frac{1}{2},\frac{1}{2})},\\
T_{(r,-r)}&= M_{(r,-r)}+qM_{(r-1,-r+1)}+qM_{(-r+1,r-1)}+q^2M_{(-r,r)} \quad (r\ge 3/2).
\end{split}
\end{equation*}
From these formulas we can obtain explicit formulas for the bar involution on $\widehat{\TT}^2$:
\begin{equation*}
\begin{split}
\overline{M_{(r,s)}}&=M_{(r,s)}\quad (r\le s, r+s\not=0),\\
\overline{M_{(r,s)}}&=M_{(r,s)}+(q-q^{-1})M_{(s,r)}\quad (r> s, r+s\not=0),\\
\overline{M_{(-r,r)}}&=M_{(-r,r)} + \sum_{s>r}(q-q^{-1})(-q)^{r+1-s}M_{(-s,s)}\quad (r>0),\\
\overline{M_{(r,-r)}}&=M_{(r,-r)} + q(q-q^{-1})M_{(-r,r)}
+\sum_{0<s<r} (q-q^{-1})(-q)^{s+1-r}M_{(s,-s)}\\ &+ \sum_{s<0}q^{-1}(q-q^{-1})(-q)^{s+1-r}M_{(s,-s)} \quad (r>0).
\end{split}
\end{equation*}
\end{example}

One could develop the crystal structure for $\TT^n$ and an algorithm of computing the canonical basis $T_\la$ similar to \cite{Br2}.

%
%
%

\subsection{Canonical basis of $\widehat{\FF}^n$}\label{subsec:CB of Fock C}

We shall introduce canonical basis of $\widehat{\FF}^n$ and its dual basis (compare~\cite[Sections 3.2 and 3.3]{Br2}).

We first note that $\widehat{\KK}^n$ is invariant under the bar involution on  $\widehat{\TT}^n$ (cf.~\cite[Lemma 3.4]{Br2}), and hence there exists a bar involution $\overline{\phantom{x}}$ on $\widehat{\FF}^n$ induced from the one on $\widehat{\TT}^n$.
For $\la\in \La_{\hZ}^+$, we have by \eqref{eq:bar on M}
\begin{equation}\label{eq:bar on F_lambda}
\ov{F_\la} \in  F_\la + \sum_{\mu\succ \la} \Z[q,q^{-1}] F_\mu.
\end{equation}
Hence there exists a unique bar-invariant topological $\Q(q)$-basis $\{\,U_\la\,|\,\la\in\La_{\hZ}^+\,\}$ of $\widehat{\FF}^n$ such that
\begin{equation}\label{eq:CB on Fock}
U_\la=\sum_{\mu\in\La_{\hZ}^+}{u_{\mu\la}(q)}F_\mu,
\end{equation}
where $u_{\la\la}(q)=1$, $u_{\mu\la}(q)\in q\Z[q]$ for $\mu\succ\la$, and $u_{\mu\la}(q)=0$ otherwise.
We call $\{\,U_\la\,|\,\la\in\La_{\hZ}^+\,\}$ the {\em canonical basis} of $\widehat{\FF}^n$.
We can show as in \cite[Lemma 3.8]{Br2} that
\begin{align}\label{can:in:quot}
\pi(T_{w_0\la})=\begin{cases}
U_\la,& \text{ if }\la\in\La_{\hZ}^+,\\
0,& \text{ otherwise},
\end{cases}
\end{align}
which in particular implies that $\{\,T_{w_0\la}\,|\,\la\in \La_{\hZ}\setminus\La_{\hZ}^+\,\}$ is a topological $\Q(q)$-basis of $\widehat{\KK}^n$.

Now, we consider the $\Q(q)$-vector space
\begin{equation*}
\EE^n:=\bigoplus_{\la\in \La_{\hZ}^+}\Q(q)E_\la,
\end{equation*}
where
\begin{align}\label{formula:E}
E_\la:=\sum_{\mu\in\La_{\hZ}^+} u_{-w_0\la\,-w_0\mu}(q^{-1})L_\mu,\quad \text{ for $\la\in \La_{\hZ}^+$}.
\end{align}
Indeed, \eqref{formula:E} is a finite linear combination, and hence $\{\,L_\la\,|\,\la\in\La_{\hZ}^+\,\}$ is a $\Q(q)$-basis of $\EE^n$. Let $\widehat{\EE}^n$ be the closure of $\EE^n$ in $\widehat{\TT}^n$.
Since $\{\,T_{w_0\la}\,|\,\la\in\La_\hZ\setminus\La^+_\hZ\,\}$ is a topological $\Q(q)$-basis of $\widehat{\KK}^n$ by \eqref{can:in:quot}, it follows from \eqref{eq:TL pairing} that the bilinear form \eqref{bi:form} induces a pairing $\langle\cdot,\cdot \rangle$ between $\widehat{\EE}^n$ and $\widehat{\FF}^n$ such that
\begin{equation}\label{eq:dual pairing}
\langle L_{-w_0\la},U_\mu\rangle = \delta_{\la\mu}, \quad \text{ for } \la, \mu\in\La^+_\hZ.
\end{equation}
Then similar to \cite[Lemma 3.15]{Br2} we see that the pairing \eqref{eq:bilinear form on Tn} satisfies
\begin{equation}\label{eq:pairing E and M}
\left( E_{\la},M_\mu \right) = \delta_{\la\mu},   \quad \text{ for }\la, \mu\in\La^+_\hZ,
\end{equation}
Using this, we can prove as in \cite[Theorem 3.16]{Br2} that $\{\,L_\la\,|\,\la\in\La_{\hZ}^+\,\}$ is the unique bar-invariant topological $\Q(q)$-basis of $\widehat{\EE}^n$ such that ${L_\la}\in E_\la+\sum_{\mu\in\La^+_\hZ}q^{-1}\Z[q^{-1}]E_\mu$ (possibly an infinite sum). Moreover, we have
\begin{align}\label{formula:U}
L_\la=\sum_{\mu\in\La_{\hZ}^+}\ell_{\mu\la}(q)E_\mu, \quad \text{ for $\la\in \La_{\hZ}^+$},
\end{align}
where $\ell_{\mu\la}(q)$ are given in \eqref{can:dual:can}.

The following formula can be obtained by using \eqref{eq:pairing E and M} and an analogue of \cite[Lemma 2.11]{Br2}, which implies that $\EE^n$ is a $\UU$-module.
\begin{lem}\label{lem:action of U on E-basis}
For $\la\in\La^+_\hZ$ and $i\geq 0$, we have
$ 
E_iE_\la =\sum_{\mu}q^{m_\mu}E_\mu,
$ 
where the sum is over all $\mu\in\Lambda^+_\hZ$
with $m_\mu=-(\alpha_i,\delta_{\la_{r+1}}+\ldots+\delta_{\la_n})$ satisfying $\mu=\la-\varepsilon_r$, for some $1\leq r\leq n$ and ${\rm wt}(\mu)={\rm wt}(\la)+\alpha_i$.
Also we have
$ 
F_iE_\la =\sum_{\nu}q^{n_\nu}E_\nu,
$ 
where the sum is over all $\nu\in\Lambda^+_\hZ$
with $n_\nu=(\alpha_i,\delta_{\la_{1}}+\ldots+\delta_{\la_{r-1}})$ satisfying $\nu=\la+\varepsilon_r$, for some $1\leq r\leq n$ and ${\rm wt}(\nu)={\rm wt}(\la)-\alpha_i$.
\end{lem}

\section{Comparison of canonical bases of types $A$ and $C$}\label{sec:CBabc}

\subsection{Canonical basis of a mixed tensor}\label{subsec:CB of mixed tensor}
Let $\VV=\VV_+\oplus \VV_-$, where
\begin{equation*}
\VV_\pm:=\bigoplus_{r\in\hZ_+}\Q(q) v_{\pm r}.
\end{equation*}
Note that $\VV_+$ is a natural representation of $\UU^{\mf a}\subseteq \UU$, and $\VV_-$ is isomorphic to the restricted dual of $\VV_+$.
Fix $n\geq 1$. Let ${\bf s}=({\rm s}_1,\ldots,{\rm s}_n)$ be such that ${\rm s}_i\in\{+,-\}$ (or $\{1,-1\}$), and
\begin{equation}
 \label{TTs}
\TT^{n}_{\bf s}:=\VV_{{\rm s}_1}\otimes\cdots\otimes\VV_{{\rm s}_n}.
\end{equation}
For $\la\in\La_{\hZ}$ with  $M_\la\in \TT^n_{\bf s}$, we write
\begin{equation}\label{eq:sgn(lambda)}
{\rm sgn}(\la)={\bf s}.
\end{equation}
As a $\UU^{\mf a}$-module, we have $\TT^{n}=\bigoplus_{\bf s}\TT^{n}_{\bf s}$. 
Let $\widehat{\TT}^n_{\bf s}$ be the closure of $\TT^n_{\bf s}$ in $\widehat{\TT}^n$, which has a topological $\Q(q)$-basis $\{\,M_\la\,|\,\la\in \La_{\hZ},\ {\rm sgn}(\la)={\bf s}\,\}$. 

Let $\psi$ denote the bar involution on $\widehat{\TT}^n$ in Section \ref{subsec:canonical basis of Tn}. We define an involution $\psi_{\mf a}$ on $\widehat{\TT}^n_{\bf s}$  in the same way as in $\psi$ by taking the limit of the involution $\psi^{(k)}_{\mf a}$ on $\pi_k(\TT^n_{\bf s})$ with respect to the quasi-$\mc R$-matrix $\Theta^{(k)}_{\mf a}$ in \eqref{eq:Theta_J}.
Then for $\la\in\La_{\hZ}$ with ${\rm sgn}(\la)={\bf s}$, we have by \eqref{eq:Theta_J} and \eqref{eq:bar on M},
\begin{equation*}
\psi_{\mf a}(M_\la)=M_\la + \sum_{\substack{\mu\prec \la \\ {\rm sgn}(\mu)={\bf s}}}r^{\mf a}_{\mu\la}(q)M_\mu,
\end{equation*}
for some $r^{\mf a}_{\mu\la}(q)\in \Z[q,q^{-1}]$. Hence,
there exists a unique topological $\Q(q)$-basis $\{\,T^{\mf a}_\la\,|\,\la\in \La_{\hZ},\ {\rm sgn}(\la)={\bf s}\,\}$ of $\widehat{\TT}^n_{\bf s}$ such that $\psi_{\mf a}(T^{\mf a}_\la)=T^{\mf a}_\la$ and
\begin{equation*}\label{eq:CB T type A}
\begin{split}
T^{\mf a}_\la=\sum_{\substack{\mu\in\La_\hZ \\ {\rm sgn}(\mu)={\bf s}}}t^{\mf a}_{\mu\la}(q)M_\mu,\\
\end{split}
\end{equation*}
with $t^{\mf a}_{\la\la}(q)=1$, $t^{\mf a}_{\mu\la}(q)\in q\Z[q]$ for $\mu\prec\la$, and $t^{\mf a}_{\mu\la}(q)=0$ otherwise.

Let $l, m\in\Z_+$ such that $l+m=n$.  Let $\Sigma_{l,m}$ be the set of ${\bf s}=({\rm s}_1,\ldots,{\rm s}_n)\in\{+,-\}^n$ that contains $l$ $-$'s and $m$ $+$'s. The symmetric group $\mf{S}_n$ acts on $\{+,-\}^n$ as place permutations, and the orbit
$$\mf{S}_n\cdot(\underbrace{-,\ldots,-}_l,\underbrace{+,\ldots,+}_m)$$
is in one-to-one correspondence with $\Sigma_{l,m}$, and also with the set of coset representatives of minimal length in $\mf{S}_n/\mf{S}_l\times \mf{S}_m$. Therefore, the Bruhat ordering on  $\mf{S}_n/\mf{S}_l\times \mf{S}_m$ induces a partial ordering $\leqq$ on $\Sigma_{l,m}$.
For example, exchanging a pair $(+,-)$ in ${\bf s}\in \Sigma_{l,m}$ with $(-,+)$ gives
\begin{equation}\label{eq:example of Bruhat on A}
{\bf s}'=(\ldots, - , \ldots, + ,\ldots) \leqq {\bf s}=(\ldots, + , \ldots, - ,\ldots),
\end{equation}
so that
\begin{equation}
 \label{s:min}
{\bf s}_{\rm min} ={\bf s}_{\rm min}(l,m) :=(\underbrace{-,\ldots,-}_l,\underbrace{+\ldots,+}_m)\ \leqq \
{\bf s}_{\rm max}:=(\underbrace{+,\ldots,+}_m,\underbrace{-\ldots,-}_l),
\end{equation}
where ${\bf s}_{\rm min}$ and ${\bf s}_{\rm max}$ are the smallest and largest elements in $\Sigma_{l,m}$ with respect to $\leqq$, respectively.

Now choose any total ordering $\le$ on $\Sigma_{l,m}$ which is compatible with $\leqq$, that is, ${\bf s}\leqq{\bf t}$ implies that ${\bf s}\le{\bf t}$, for ${\bf s}, {\bf t}\in \Sigma_{l,m}$. For each ${\bf s}\in\{+,-\}^n$, we define
\begin{align*}
\widehat{\TT}^n_{\le \bf s}:=\bigoplus_{\bf t\le s}\widehat{\TT}^{n}_{\bf t},\quad\quad
\widehat{\TT}^n_{< \bf s}:=\bigoplus_{\substack{\bf t\leq s \\ \bf t\neq s}}\widehat{\TT}^{n}_{\bf t},
\end{align*}
so that we have $\widehat{\TT}^n_{\le \bf s}\subseteq\widehat{\TT}^n_{\le \bf t}$, for ${\bf s\le t}$.
It follows by definition of $\psi$ and \eqref{eq:example of Bruhat on A}
 that $\psi(\widehat{\TT}^n_{\le \bf s})\subseteq \widehat{\TT}^n_{\le \bf s}$, for all ${\bf s}\in \Sigma_{l,m}$.
Hence we have a filtration of $\psi$-invariant subspaces
\begin{align}\label{eq:filtration}
0\subsetneq \widehat{\TT}^n_{\le \bf s_{\rm min}}\subsetneq\cdots \subsetneq  \widehat{\TT}^n_{\le \bf s} \subsetneq\cdots \subsetneq  \widehat{\TT}^n_{\le \bf s_{\rm max}},
\end{align}
and furthermore, we have a $\UU^{\mf a}$-module isomorphism
$\widehat{\TT}^n_{\le \bf s}/\widehat{\TT}^n_{< \bf s}\cong \widehat{\TT}^n_{\bf s}.$

\begin{prop}\label{thm:head:can:basis}
Let ${\bf s}\in \Sigma_{l,m}$. For  $\la\in \La_{\hZ}$ with ${\rm sgn}(\la)={\bf s}$, we have
\begin{equation*}
\pi_{\bf s}(T_\la)=T^{\mf a}_\la,
\end{equation*}
where $\pi_{\bf s}:\widehat{\TT}^{n}\rightarrow \widehat{\TT}^{n}_{\bf s}$ is the canonical projection.
\end{prop}

\begin{proof}
Let $M_\la\in \TT^n_{\bf s}$ be so that $M_\la+\widehat{\TT}^n_{< \bf s}\in \widehat{\TT}^n_{\le \bf s}/\widehat{\TT}^n_{< \bf s}$.
To prove the proposition, it is enough to show the following:
\begin{align}\label{eq:psi=psiJ}
\psi(M_\la+\widehat{\TT}^n_{< \bf s})=\psi_{\mf a}(M_\la)+ \widehat{\TT}^n_{< \bf s}.
\end{align}
Indeed, the proof of \eqref{eq:psi=psiJ} reduces to the case of finite rank. So we may assume that $\psi$, $\psi_{\mf a}$, and $\widehat{\TT}^n_{< \bf s}$ are $\psi^{(k)}$, $\psi_{\mf a}^{(k)}$, and $\pi_k(\TT^n_{< \bf s})$, respectively, for $M_\la\in \TT^n_{< k}$. Below we use a shorthand notation to denote $\Theta_\nu(u\otimes v)=\Theta^+_\nu(u)\otimes \Theta^-_\nu(v)$ for $\nu\in Q^+$ and $u, v\in \TT$.

We proceed by induction on $n$. The claim is clear when $n=1$ since the elements $v_r$ ($r\in\hZ$) are canonical basis elements for both $\psi$ and $\psi_{\mf a}$.
We assume that $n\geq 2$.

We observe that $\UU^{\mf a}\VV_\pm\subseteq\VV_\pm$, and $F_0\VV_+=E_0\VV_-=0$. Also,
if $\nu\in Q^+\setminus Q^+_{\mf a}$, then $\Theta^+_\nu$ (respectively, $\Theta^-_\nu$) is a linear combination of product of $E_i$'s (respectively, $F_i$'s) including a factor $E_0$ (respectively, $F_0$) by \eqref{eq:PBW vector} and \eqref{eq:Theta_nu}.

First suppose that ${\rm s}_n=+$. Set ${\la'}=(\la_1,\ldots,\la_{n-1})$. We have by induction hypothesis and the observations in the above paragraph
\begin{align*}
\psi(M_\la)=&\sum_{\nu\in Q^+}\Theta^+_{\nu}(\psi(M_{{\la'}}))\otimes\Theta^-_{\nu}(v_{\la_n})=
\sum_{\nu\in Q_{\mf a}^+}\Theta^+_\nu(\psi(M_{{\la'}}))\otimes\Theta^-_{\nu}(v_{\la_n})\\
=&\sum_{\nu\in Q_{\mf a}^+}\Theta^+_\nu(\psi_{\mf a}(M_{{\la'}}))\otimes\Theta^-_{\nu}(v_{\la_n}) +\sum_{\nu\in Q_{\mf a}^+}\sum_{\eta} c_\eta\Theta^+_\nu(M_{\eta})\otimes\Theta^-_{\nu}(v_{\la_n})\\
=&\psi_{\mf a}(M_{\la}) +\sum_{\nu\in Q_{\mf a}^+}\sum_\eta c_\eta\Theta^+_{\nu}(M_{\eta})\otimes\Theta^-_{\nu}(v_{\la_n}).
\end{align*}
Here $c_\eta\in\Z[q,q^{-1}]$ and $M_\eta\in\TT^{n-1}_{< \bf s'}$ with  ${\bf s'}= ({\rm s}_1,\ldots,{\rm s}_{n-1})$, and hence the second summand in the last equation lies in $\TT^n_{< \bf s}$.

Next suppose that ${\rm s}_n=-$. Choose $n'<n$ to be minimal such that  ${\rm s}_{j}=-$ for all $n'<j\le n$. Denote ${\bf s'}= ({\rm s}_1,\ldots,{\rm s}_{n'})$,
${\la'}=(\la_1,\ldots,\la_{n'})$ and ${\la''}=(\la_{n'+1},\ldots,\la_n)$. We have by induction hypothesis
\begin{align*}
\psi(M_\la)
=&\sum_{\nu\in Q^+}\Theta^+_\nu(\psi(M_{{\la'}}))\otimes\Theta^-_{\nu}(\psi(M_{\la''}))
= \sum_{\nu\in Q^+}\Theta^+_\nu(\psi(M_{{\la'}}))\otimes\Theta^-_{\nu}(\psi_{\mf a}(M_{\la''}))\\
=&\sum_{\nu\in Q^+}\Theta^+_\nu(\psi_{\mf a}(M_{{\la'}}))\otimes\Theta^-_{\nu}(\psi_{\mf a}(M_{\la''}))
+\sum_{\nu\in Q^+}\sum_{\zeta <\bf s'} c_\zeta \Theta^+_\nu(M_{\zeta})\otimes\Theta^-_{\nu}(\psi_{\mf a}(M_{\la''}))\\
=&\sum_{\nu\in Q_{\mf a}^+}\Theta^+_\nu(\psi_{\mf a}(M_{{\la'}}))\otimes\Theta^-_{\nu}(\psi_{\mf a}(M_{\la''}))+\sum_{\nu\in Q^+\setminus Q^+_{\mf a}}\Theta^+_\nu(\psi_{\mf a}(M_{{\la'}}))\otimes\Theta^-_{\nu}(\psi_{\mf a}(M_{\la''}))\\
&+\sum_{\nu\in Q^+}\sum_{\zeta<\bf s'} c_\zeta\Theta^+_\nu(M_{\zeta})\otimes\Theta^-_{\nu}(\psi_{\mf a}(M_{\la''})).
\end{align*}
where $c_\zeta\in\Z[q,q^{-1}]$ and $M_\zeta\in\TT^{n'}_{< \bf s'}$ with ${\bf s'}= ({\rm s}_1,\ldots,{\rm s}_{n'})$. By using \eqref{eq:example of Bruhat on A}, we see that the second and the third summands above lie in $\TT^n_{< \bf s}$, while the first summand is $\psi_{\mf a}(M_\la)$. This completes the induction.
\end{proof}

\begin{cor}\label{cor:ACB=CCB}
If ${\bf s}={\bf s}_{\rm min}\in \Sigma_{l,m}$, then we have $\psi(v)=\psi_{\mf a}(v)$, for $v\in \widehat{\TT}^n_{\bf s}$,
and
$
T_\la=T^{\mf a}_\la,
$ 
for $\la\in \La_{\hZ}$ with ${\rm sgn}(\la)={\bf s}$.
\end{cor}
\begin{proof}
It follows directly from \eqref{eq:filtration} and \eqref{eq:psi=psiJ}.
\end{proof}

\begin{rem}
The statements of Proposition \ref{thm:head:can:basis} and Corollary \ref{cor:ACB=CCB} for $n=2$ actually can be observed from Example \ref{ex:CB:n=2}.
\end{rem}

%
%

\subsection{Canonical basis of the $q$-wedge spaces of types $A$ and $C$}\label{subsec:CBA=CBC}

Fix $n\geq 1$ and let $l, m\in\Z_+$ such that $l+m=n$.
Let
\begin{equation*}
\FF_{(l,m)}:=\bigoplus_{\substack{\la\in\La_{\hZ}^+ \\ {\rm sgn}(\la)\in \Sigma_{l,m}}}\Q(q)F_\la,
\end{equation*}
and let $\widehat{\FF}_{(l,m)}$ be its closure in $\widehat{\FF}^n$. Note $\FF^n =\bigoplus_{l+m=n} \FF_{l,m}$.

\begin{lem}
The set $\{\,U_\la\,|\, \la\in\La_{\hZ}^+,\ {\rm sgn}(\la)\in \Sigma_{l,m}\,\}$ is a topological $\Q(q)$-basis of $\widehat{\FF}_{(l,m)}$.
\end{lem}

\begin{proof}
By \eqref{eq:CB on Fock} and Proposition \ref{prop:Bruhat PS=Br}, we see that ${\rm sgn}(\mu)\in \Sigma_{l,m}$ for $\mu\in\La^+_\hZ$ with $\mu\succ \la$. Hence, $U_\la\in \widehat{\FF}_{(l,m)}$ and $\{\,U_\la\,|\, \la\in\La_{\hZ}^+,\ {\rm sgn}(\la)\in \Sigma_{l,m}\,\}$ is a topological $\Q(q)$-basis of $\widehat{\FF}_{(l,m)}$.
\end{proof}

We shall show that $\{\,U_\la\,|\, \la\in\La_{\hZ}^+ \}$ is indeed a canonical basis of type $A$.
To that end, we define the $q$-wedge spaces of type $A$:
\begin{equation*}
\bigwedge^k \VV_+ := \VV_+^{\otimes k}/\KK^k_+,\quad\quad
\bigwedge^k \VV_- := \VV_-^{\otimes k}/\KK^k_-,\quad (k\geq 0),
\end{equation*}
where $\KK^k_\pm=\KK^k \cap \VV_\pm^{\otimes k}$. Note that $\bigoplus_{k\geq 0}\KK^k_\pm$ is the two-sided ideal in the tensor algebra $\bigoplus_{k\geq 0}\VV_\pm^{\otimes k}$ generated by $v_r\otimes v_r$ and $v_r\otimes v_s +q v_s\otimes v_r$ for $r, s\in \pm\hZ_+$ with $r>s$, which is invariant under the action of $\UU^{\mf a}$, and hence $\bigwedge^k \VV_\pm$ are well-defined $\UU^{\mf a}$-modules.
We may also regard $\bigwedge^k \VV_\pm$ as a subspace of $\FF^k$ by definition.

We let
\begin{equation}\label{eq:Fock space A}
\FF^{\mf a}_{(l,m)}:=\bigwedge^l \VV_- \otimes \bigwedge^m \VV_+.
\end{equation}
For $\la\in\La_{\hZ}^+$ with ${\rm sgn}(\la)\in \Sigma_{l,m}$, put
\begin{equation*}
F^{\mf a}_\la := F_{\la^-}\otimes F_{\la^+},
\end{equation*}
where $\la^+=(\la_1,\ldots,\la_l)$ and $\la^-=(\la_{l+1},\ldots,\la_n)$.
The following can be verified directly.

\begin{lem}\label{lem:Fa=F}
The set $\{\,F^{\mf a}_\la\,|\,\la\in\La_{\hZ}^+, \ {\rm sgn}(\la)\in \Sigma_{l,m}\}$ forms a $\Q(q)$-basis of $\FF_{(l,m)}^{\mf a}$, and there exists
an isomorphism of  $\UU^{\mf a}$-modules $\FF_{(l,m)}^{\mf a}\rightarrow \FF_{(l,m)}$
sending $F^{\mf a}_\la$ to $F_\la$, for $\la\in\La_{\hZ}^+$ with ${\rm sgn}(\la)\in \Sigma_{l,m}$.
\end{lem}

By Lemma \ref{lem:Fa=F}, we can take a completion $\widehat{\FF}^{\mf a}_{(l,m)}$ of $\FF^{\mf a}_{(l,m)}$ corresponding to $\widehat{\FF}_{(l,m)}$, and a $\Q(q)$-linear isomorphism
$\rho : \widehat{\FF}^{\mf a}_{(l,m)} \longrightarrow \widehat{\FF}_{(l,m)}
$
given by $\rho(F^{\mf a}_\la)=F_\la$, for $\la\in\La_{\hZ}^+$ with ${\rm sgn}(\la)\in \Sigma_{l,m}$.
We can assemble these linear isomorphisms for various $l,m$ into one $\Q(q)$-linear isomorphism
$$
\rho: \widehat{\FF}^{n, \mf a} := \bigoplus_{l+m=n}\widehat{\FF}^{\mf a}_{(l,m)} \longrightarrow \widehat{\FF}^n.
$$

In a similar way as in Section \ref{subsec:CB of mixed tensor}, we define an involution on $\widehat{\FF}^{\mf a}_{(l,m)}$, still denoted by $\psi_{\mf a}$, that is, $\psi_{\mf a}(u\otimes v)=
\lim_{k\rightarrow \infty}\Theta^{(k)}_{\mf a}(\psi^{(k)}_{\mf a}(u)\otimes \psi^{(k)}_{\mf a}(v))$, for $u\in \bigwedge^l\VV_-$ and $v\in \bigwedge^m\VV_+$. Here $\psi^{(k)}_{\mf a}(u)$ and $\psi^{(k)}_{\mf a}(v)$ denote the involutions on $\bigwedge^l\pi_k(\VV_-)$ and $\bigwedge^m\pi_k(\VV_+)$ induced from those on $\pi_k(\VV^{\otimes l}_-)$ and $\pi_k(\VV^{\otimes m}_+)$ with respect to $\Theta^{(k)}_{\mf a}$, which by \eqref{eq:psi=psiJ} coincide with $\psi^{(k)}(u)$ and $\psi^{(k)}(v)$ on $\pi_k({\TT}^l)$ and $\pi_k({\TT}^m)$, respectively. By \eqref{eq:Theta_J} and \eqref{eq:bar on F_lambda}, we have
$ 
\psi_{\mf a}(F^{\mf a}_\la) \in F^{\mf a}_\la + \sum_{\mu\succ \la} \Z[q,q^{-1}] F^{\mf a}_\mu.
$ 
Hence there exists a unique topological $\Q(q)$-basis $\{\,U^{\mf a}_\la\,|\,\la\in\La_{\hZ}^+, {\rm sgn}(\la)\in \Sigma_{l,m}\,\}$ of $\widehat{\FF}^{\mf a}_{(l,m)}$ such that $\psi_{\mf a}(U^{\mf a}_\la)=U^{\mf a}_\la$ and
\begin{equation*}
U^{\mf a}_\la=\sum_{\mu\in\La_{\hZ}^+}{u^{\mf a}_{\mu\la}(q)}F^{\mf a}_\mu,
\end{equation*}
with $u^{\mf a}_{\la\la}(q)=1$, $u^{\mf a}_{\mu\la}(q)\in q\Z[q]$ for $\mu\succ\la$, and $u^{\mf a}_{\mu\la}(q)=0$ otherwise.

\begin{thm}\label{thm:ACB=CCB}
The isomorphism $\rho: \widehat{\FF}^{n, \mf a} \rightarrow \widehat{\FF}^n$ preserves the canonical bases, that is,
\begin{equation*}
\rho(U^{\mf a}_\la) = U_\la,  \quad \text{ for } \la\in\La_{\hZ}^+.
\end{equation*}
In particular, we have $u^{\mf a}_{\mu\la}(q)=u_{\mu\la}(q)$, for $\la, \mu \in\La_{\hZ}^+$.
\end{thm}
\begin{proof}
The involution $\psi_{\mf a}$ on $\widehat{\TT}^n_{\bf s}$ induces an involution on $\widehat{\FF}_{(l,m)}$ still denoted by $\psi_{\mf a}$, which defines the $\psi_{\mf a}$-invariant canonical basis of $\widehat{\FF}_{(l,m)}$, say, $\{\, U^{'\mf a}_\la \,|\,\la\in\La_{\hZ}^+,\ {\rm sgn}(\la)\in \Sigma_{l,m}\,\}$.

Let $\la\in\La_{\hZ}^+$ with ${\rm sgn}(\la)\in \Sigma_{l,m}$.
We have $\pi(T_{w_0\la})=U_{\la}$ by \eqref{can:in:quot}, and also $\pi(T^{\mf a}_{w_0\la})= U^{'\mf a}_\la$ by the same arguments.
Since ${\rm sgn}(w_0\la)$ is minimal in $\Sigma_{l,m}$, we have $T_{w_0\la}=T^{\mf a}_{w_0\la}$ by Corollary \ref{cor:ACB=CCB} and hence $U_{\la}=U^{'\mf a}_\la$.
On the other hand, $\rho$ commutes with $\psi_{\mf a}$ and preserves the Bruhat ordering $\succeq$ by Lemma \ref{lem:Fa=F}. Hence we have $\rho(U^{\mf a}_\la)=U^{'\mf a}_\la=U_\la$ by the uniqueness of the $\psi_{\mf a}$-invariant canonical basis of $\widehat{\FF}_{(l,m)}$.
\end{proof}

\begin{rem}
Theorem~\ref{thm:ACB=CCB} suggests some possible new connections between
representation theory of $\qn$ (in blocks ``of signature $\Sigma_{l,n-l}$") and representation theory of $\mathfrak{gl}(l|n-l)$, which will be very interesting to formulate precisely.
In case of $l=0$ or $l=n$, there are such connections developed by Frisk and Mazorchuk (see \cite{FM}).

A finite rank version of Theorem~\ref{thm:ACB=CCB} also holds for $\UU_q(\mf c_k)$ in place of $\UU =\UU_q(\mf c_\infty)$.
\end{rem}

\section{Kazhdan-Lusztig theory for  queer Lie superalgebra $\mf{q}(n)$}\label{sec:O}

\subsection{Representations of $\mf q(n)$}
We assume that the base field is $\C$.
Let $\mf{gl}(n|n)$ be the general linear Lie superalgebra of the complex superspace $\C^{n|n}$. Let $\{u_{\ov{1}},\ldots,u_{\ov{n}}\}$ be an ordered basis for the even subspace $\C^{n|0}$, and let  $\{u_1,\ldots,u_n\}$ be an ordered basis for the odd subspace $\C^{0|n}$ so that $\gl(n|n)$ is realized as $2n \times 2n$ complex matrices indexed by $I(n|n):=\{\,\ov{1}<\ldots<\ov{n}<1<\ldots<n\,\}$.
The subspace
\begin{equation}\label{eq:q(n)}
\G \equiv {\mf q}(n):=\left\{\,
\begin{pmatrix}
A & B \\
B & A
\end{pmatrix}
\,\Big\vert\,\text{$A$, $B$ : $n \times n$ matrices}\,\right\}
\end{equation}
forms a subalgebra of $\gl(n|n)$ called the {\em queer Lie superalgebra}.

Let $E_{ab}$ be the elementary matrix in $\gl(n|n)$ with $(a,b)$-entry $1$ and other entries $0$, for $a,b\in I(n|n)$. Then $\{\,e_{ij},\ov{e}_{ij}\,|\,1\leq i,j\leq n\,\}$ is  a linear basis of $\G$, where $e_{ij}=E_{\bar{i}\bar{j}}+E_{ij}$ and $\ov{e}_{ij}=E_{\bar{i}j}+E_{i\bar{j}}$.  The even subalgebra $\G_{\even}$ is spanned by $\{\,e_{ij}\,|\,1\leq i,j\leq n\,\}$, and isomorphic to the general linear Lie algebra $\gl(n)$.

Let $\h=\h_\even\oplus\h_\odd$ be the standard Cartan subalgebra of $\G$, with linear bases $\{\,h_i:=e_{ii}\,|\,\,1\leq i\leq n\,\}$ and $\{\,\ov{h}_i:=\ov{e}_{ii}\,|\,1\leq i\leq n\,\}$ of $\h_\even$ and $\h_\odd$, respectively.
Let $\{\,\varepsilon_i\,|\,1\leq i\leq n\,\}$ be the basis of $\h^\ast_{\even}$ dual to $\{\,h_i\,|\,1\leq i\leq n\,\}$. We define a symmetric bilinear form $(\ ,\ )$ on $\h_\even^\ast$ by $(\varepsilon_i,\varepsilon_j)=\delta_{ij}$, for $1\leq i,j\leq n$. We often identify $\h^*_\even$ with $\C^n$ by regarding $\varepsilon_i$ as the standard basis element in $\C^n$, for $1\leq i\leq n$.

Let $\mf b$ be the standard Borel subalgebra of $\G$, which consists of matrices of the form  \eqref{eq:q(n)} with $A$ and $B$ upper triangular.
Let $\Phi^+$ and $\Phi^-$ be the sets of positive and negative roots with respect to $\h_\even$, respectively. We have
$\Phi^-=-\Phi^+$, and $\Phi^+=\Phi^+_\even \sqcup \Phi^+_\odd$,
where $\Phi^+_{\even}$ and $\Phi^+_{\odd}$ denote the sets of positive even and odd roots, respectively. We have $\Phi^+_\even=\Phi^+_\odd =\{\,\varepsilon_i-\varepsilon_j\,|\,1\leq i<j\leq n\,\}$ ignoring the parity.

For a $\G$-module $V$ and $\mu\in \h^\ast_\even$, let $V_\mu=\{\,v\in V\,|\,h\cdot v=\mu(h)v \text{ for $h\in\h_\even$}\, \}$ denote its $\mu$-weight space. If $V$ has a weight space decomposition $V=\bigoplus_{\mu\in\h_\even^\ast}V_\mu$, its character is given as usual by $\text{ch}V=\sum_{\mu\in\h^*_\even}\dim V_\mu e^\mu$.

Let $\la=(\la_1,\ldots,\la_n) \in \h^*_\even$. We denote by $\ell(\la)$ the number of $i$'s with $\la_i\neq 0$.
Consider a symmetric bilinear form on $\h_\odd$ given by $\langle \cdot, \cdot \rangle_\la = \la([\cdot, \cdot])$, and let $\h'_\odd$ be a maximal isotropic subspace associated to $\langle \cdot, \cdot \rangle_\la$. Put $\h'=\h_\even \oplus \h'_\odd$. Let $\C v_\la$ be the one-dimensional $\h'$-module with $h\cdot v_\la = \la(h) v_\la$ and $h'\cdot v_\la=0$, for $h\in \h_\even$ and $h'\in \h'_\odd$. Then $W_\la:={\rm Ind}_{\h'}^\h \C v_\la$ is an irreducible $\h$-module of dimension $2^{\lceil \ell(\la)/2\rceil}$, where $\lceil\cdot\rceil$ denotes the ceiling function. We put $\Delta(\la):={\rm Ind}_{\mf b}^{\G}W_\la$ called a Verma module, where $W_\la$ is extended to a ${\mf b}$-module in a trivial way, and its character is given by
\begin{equation*}
{\rm ch}\Delta(\la)=2^{\lceil \ell(\la)/2 \rceil}
e^\la D^{-1},
\end{equation*}
where
\begin{equation*}
D:=\prod_{\alpha\in\Phi^+_{\bar 0}}(e^{\alpha/2}-e^{-\alpha/2})\Big/ \prod_{\alpha\in\Phi^+_{\bar 1}}(e^{\alpha/2}+e^{-\alpha/2}).
\end{equation*}

We define $L(\la)$ to be the unique irreducible quotient of $\Delta(\la)$. Note that it is an $\h_\even$-semisimple $\mf b$-highest weight $\G$-module with highest weight $\la$.
We say that $\la$ is {\em atypical} if  $\la_i+\la_j= 0$ for some $1\le i<j \le n$, and {\em typical} otherwise \cite{Pe}.

\subsection{Categories of $\G$-modules}

Let $\mc O_{n}$ denote the BGG category of finitely generated $\G$-modules which are locally finite $\mf b$-modules and semisimple $\h_\even$-modules. The set $\{\,L(\la)\,|\,\la\in\h^*_\even\,\}$ is the complete set of irreducible $\G$-modules in $\mc O_n$ up to isomorphism.

Let $\mc O^{\Delta}_{n}$ be the full subcategory of $\mc O_{n}$ consisting of $M$ having a Verma flag of finite length, that is, a filtration $0=M_0\subseteq M_1 \subseteq \cdots \subseteq M_r=M$ such that $M_i/M_{i-1}$ is a Verma module for $1\leq i\leq r$.
We denote by $(M:\Delta(\la))$ the multiplicity of $\Delta(\la)$ in a filtration of $M$. Let $T(\la)$ be the tilting module of $\mf g$ associated to $\la\in\h^*_\even$, which is the unique indecomposable module in $\mc O_{n}$ such that
(i) ${\rm Ext}_{\mc O_{n}}(\Delta(\mu),T(\la))=0$ for all $\mu\in\h^*$, (ii) $T(\la)$ has a Verma flag with $\Delta(\la)$ at the bottom. Indeed, we have $T(\la)\in \mc O^\Delta_{n}$. (See \cite[Example 7.10]{Br3} for more details.)

Let $\mc E_n$ be the category of finite-dimensional $\G$-modules. Let $\La^+$ be the set of $\la=(\la_1,\ldots,\la_n)\in\h^*_\even$ such that $\la_i\geq \la_{i+1}$ and $\la_i=\la_{i+1}$ implies $\la_i=\la_{i+1}=0$ for $1\leq i\leq n-1$. Then for $\la\in\h^*_\even$, we have $L(\la)\in \mc E_n$ if and only if $\la\in \La^+$ (see \cite[Theorem 4]{Pe}, \cite[Theorem 2.18]{CW}).

We denote by $K(\mc C)$ the Grothendieck group of a module category $\mc C$ spanned by  the equivalence classes $[M]$ for $M\in \mc C$.

For $\la=(\la_1,\ldots,\la_n)\in \La^+$, let $\mf l(\la)$ be the subalgebra spanned by $\mf h$ and the root spaces $\G_{\pm \alpha}$, for $\alpha\in \Phi^+(\la):=\{\,\varepsilon_i-\varepsilon_j\,|\,\, \la_i=\la_j\ (i< j)\,\}$, and let $\mf p(\la)=\mf l(\la) + \mf b$ be its parabolic subalgebra. Note that $\mf l(\la)=\mf h$ and $\mf p(\la)=\mf b$, for $\la\in \La^+_{\hZ}$.
Let
\begin{equation*}
E(\la):=\sum_{i\ge 0}(-1)^i\mc L_i\left({\rm Ind}_{\mf p(\la)}^\G W_\la\right)
\end{equation*}
be the Euler characteristic of ${\rm Ind}_{\mf p(\la)}^\G W_\la$,
where $W_\la$ is regarded as an irreducible $\mf l(\la)$-module, and $\mc L_i$ is the $i$th derived functor of the dual Zuckerman functor with respect to the pair $\mf l(\la)\subseteq \G$ \cite[Section 4]{San}.
In general, $E(\la)$ is a virtual $\G$-module, and its character is given by
\begin{equation}\label{eq:Euler character}
{\rm ch}E(\la)=2^{\lceil \ell(\la)/2 \rceil}D^{-1}\sum_{w\in \mathcal S_n}(-1)^{\ell(w)}w\left(\frac{e^\la}{\prod_{\beta\in \Phi^+(\la)}(1+e^{-\beta})} \right),
\end{equation}
(see \cite{PS2, CK}).
In particular, we have $E(\la)=L(\la)$, for typical $\la$.
It is well known that
\begin{equation*}
\left\{\,[E(\la)]\,\big\vert \,\la\in\La^+\,\right\},\quad \left\{\,[L(\la)]\,\big\vert \,\la\in\La^+\,\right\}
\end{equation*}
are $\Z$-bases of $K(\mc E_{n})$.
We remark that the first algorithm for computing $a_{\la\mu}$ in
\begin{equation}\label{eq:euler multiplicity}
[E(\la)]=\sum_{\mu}a_{\la\mu}[L(\mu)]\quad (\la\in\La^+),
\end{equation}
and hence an algorithm of computing the characters of all finite-dimensional irreducible $\G$-modules was developed earlier in  \cite{PS2}.

\subsection{Character formula for finite-dimensional half-integer weight modules}

For a $\G$-module $V$, we say that $V$ is an {\em integer weight module}  if it has a weight space decomposition $V=\bigoplus_{\mu\in\La_{\Z}}V_\mu$, and $V$ is an {\em half-integer weight module} if $V=\bigoplus_{\mu\in\La_{\hZ}}V_\mu$.

Let $\mc O_{n,\Z}$ and $\mc O_{n,\hZ}$ denote the full subcategories of $\mc O_n$ consisting of integer weight and half-integer weight modules, respectively. Set
\begin{equation*}
\begin{split}
\mc E_{n,\Z}:=\mc O_{n,\Z}\cap \mc E_n,\quad
\mc E_{n,\hZ}:=\mc O_{n,\hZ}\cap \mc E_n.
\end{split}
\end{equation*}
Then $L(\la)\in \mc O_{n,\Z}$ (respectively, $\mc O_{n,\hZ}$) is finite dimensional if and only if $\la\in \La^+_\Z$ (respectively, $\La^+_{\hZ}$).
The sets
\begin{equation*}
\left\{\,[X(\la)]\,\big\vert \,\la\in\La^+_\Z\,\right\},\quad \left\{\,[X(\la)]\,\big\vert \,\la\in\La^+_\hZ\,\right\}  \quad (X=E, L)
\end{equation*}
are $\Z$-bases of $K(\mc E_{n,\Z})$ and $K(\mc E_{n,\hZ})$, respectively.

We do not know the answer to the following two questions (which are closely related to and essentially equivalent to each other).

\begin{question}
 \label{q:HWC}
\begin{enumerate}
\item
Is the category $\mc E_{n,\hZ}$ a highest weight category?

\item
Is it true that $\mc L_i\big({\rm Ind}_{\mf b}^\G W_\la\big) =0$, for all $i>0$ and all $\la \in \La^+_{\hZ}$?
\end{enumerate}
\end{question}

\begin{rem}
The work \cite{F} (also cf.~\cite{FM}) shows that any (strongly) typical block in $\mc E_{n,\hZ}$ is a highest weight category.  A referee suggests that an affirmative answer to Question~\ref{q:HWC}(1) might likely follow from this by applying translation functors.
\end{rem}

Recall that Brundan gave a realization of $K(\mc E_{n,\Z})$ and a character formula for $L(\la)\in \mc E_{n,\Z}$ in terms of the (dual) canonical basis of type $B$ \cite{Br2} (see Theorem \ref{thm:Brundan Main}).
Now, let us consider the case of $\mc E_{n,\hZ}$. Put
\begin{equation*}
\begin{split}
&\EE^{n}_{\Z[q^{\pm}]}:=\bigoplus_{\la\in \La^+_\hZ}\Z[q,q^{-1}]L_\la
=\bigoplus_{\la\in \La^+_\hZ}\Z[q,q^{-1}]E_\la,\\
&\EE^{n}_{\Z}:=\Z\otimes_{\Z[q,q^{-1}]}\EE^{n}_{\Z[q^{\pm}]},
\end{split}
\end{equation*}
where $\{\,L_\la\,|\,\la\in\La^+_\hZ\,\}$ and $\{\,E_\la\,|\,\la\in\La^+_\hZ\,\}$
are the bases of the $\UU$-module $\EE^n$ given in Section \ref{subsec:CB of Fock C}, and $\Z$ is the right $\Z[q,q^{-1}]$-module with $q$ acting  as $1$.
Let $E_\la(1)=1\otimes E_\la$ and $L_\la(1)=1\otimes L_\la\in \EE^{n}_{\Z}$, for $\la\in \La^+_\hZ$.
Define a $\Z$-linear isomorphism  by
\begin{align}
\label{Psi}
\Psi_e:K(\mc E_{n,\hZ}) \longrightarrow \EE^{n}_\Z, \qquad
[E (\la)] \mapsto E_{\la}(1) \quad (\la\in\La^+_\hZ).
\end{align}

 \begin{thm}\label{thm:Main}
 We have $\Psi_e\left([L(\la)]\right)={L}_{\la}(1)$, for $\la\in\La^+_\hZ$. Equivalently, we have
\begin{equation*}
\begin{split}
[E(\la)]
&=\sum_{\mu\in\La^{+}_{\hZ},\, \mu\preceq \la}u_{-w_0\la\,-w_0\mu}(1) [L(\mu)],\\
[L(\la)]
&=\sum_{\mu\in\La^{+}_{\hZ},\, \mu\preceq \la}\ell_{\mu\la}(1) [E(\mu)].
\end{split}
\end{equation*}
\end{thm}
\begin{proof}
For $\la,\mu\in \La^+$, let $a_{\la\mu}$ be as in \eqref{eq:euler multiplicity}.
For $\la, \mu\in \La^+_\hZ$, we have
\begin{equation*}
\begin{split}
a_{\la\mu}
&=a_{\la^\natural\mu^\natural} \quad\quad\quad\quad\quad\quad\ \ \text{by \cite[Theorem 5.3]{CK}}\\
&=u_{-w_0\la^\natural\,-w_0\mu^\natural}(1) \,\quad\quad\text{by \cite[Theorem 4.52]{Br2} }\\
&=u_{-w_0\la\,-w_0\mu}(1) \quad\quad\ \ \, \text{by Corollary \ref{cor:dual CB for B = dual CB for C}.}
\end{split}
\end{equation*}
Therefore, we have $\Psi_e([L(\la)])=L_\la(1)$, and the formula for $[L(\la)]$, for $\la\in\La^+_\hZ$.
\end{proof}

\begin{rem}{\rm
Theorem \ref{thm:Main} can be also proved directly by imitating the arguments of \cite[Theorem 4.52]{Br2} for $\mc E_{n,\Z}$, without using  \cite[Theorem~ 4.52]{Br2} and \cite[Theorem~ 5.3]{CK}.
Our ``shortcut" approach here has the advantage in revealing connections between canonical bases of types $A$, $B$, and $C$, which suggests some possible connections between the associated module categories \cite[Remarks 6.5 and 6.7]{CK} (see also~\cite{FM}).}
\end{rem}

\subsection{Translation functors on $\mc O_{n,\hZ}$}

For $\la\in\h^*$, let $\chi_\la$ denote the central character of $L(\la)$ (see, e.g., \cite[Section 2.3]{CW} for more details).
The central characters given  by Sergeev \cite{Sv} admit the following characterization in terms of the weight function \eqref{eq:wt}
(cf.~\cite[Theorem 4.19]{Br2}).

\begin{lem}\label{lem:central character}
Let $\la, \mu\in \La_{\hZ}$. Then we have $\chi_\la=\chi_\mu$ if and only if ${\rm wt}(\la)={\rm wt}(\mu)$.
\end{lem}
We have by Lemma \ref{lem:central character} that
\begin{equation}\label{eq:block decomp O}
\mc O_{n,\hZ}=\bigoplus_{\gamma\in P}\mc O_\gamma,
\end{equation}
where $\mc O_\gamma$ is the full subcategory of modules in $\mc O_{n,\hZ}$ whose simple subquotient are $L(\la)$'s such that ${\rm wt}(\la)=\gamma$.
We let ${\rm pr}_\gamma : \mc O_{n,\hZ} \rightarrow \mc O_\gamma$ be the natural projection functor for $\gamma\in P$.

Let $V=\C^{n|n}$ be the natural representation of $\G$, and $V^*$ be its dual. For $r\geq 1$, let $S^r(V)$ and $S^r(V^*)$ denote the $r$th supersymmetric tensor of $V$ and $V^*$, respectively, which are irreducible $\G$-modules (see~\cite[Proposition 3.1]{CW2}).
 Given $M\in \mc O_\gamma$, we define
\begin{equation*}\label{eq:translation}
{\rm E}^{(r)}_i M:={\rm pr}_{\gamma+r\alpha_i}(M\otimes S^r(V^*)),\quad \quad
{\rm F}^{(r)}_i M:={\rm pr}_{\gamma-r\alpha_i}(M\otimes S^r(V)),
\end{equation*}
for $i\geq 0$ and $r\geq 1$. We write ${\rm E}_i:={\rm E}^{(1)}_i$ and ${\rm F}_i:={\rm F}^{(1)}_i$, for $i\geq 0$.
By extending additively according to \eqref{eq:block decomp O}, we obtain exact functors ${\rm E}^{(r)}_i, {\rm F}^{(r)}_i : \mc O_{n,\hZ}\rightarrow \mc O_{n,\hZ}$ called the {\em translation functors}.

Let $\mc O^{\Delta}_{n,\hZ}:=\mc O^{\Delta}_{n}\cap\mc O_{n,\hZ}$.
The following lemma implies that $\mc O^{\Delta}_{n,\hZ}$ is invariant under ${\rm E}^{(r)}_i$ and ${\rm F}^{(r)}_i$, for $i\geq 0$ and $r\geq 1$.

\begin{lem}\label{lem:translation}
Let $\la\in \La_\hZ$. For $i\geq 0$ and $r\geq 1$,
we have
\begin{itemize}

\item[(1)] ${\rm E}^{(r)}_i \Delta(\la)$ has a filtration whose non-zero subquotients
consist of Verma modules
$\Delta(\la - \varepsilon_{j_1}-\cdots-\varepsilon_{j_r})$ for $1\leq j_1<\cdots<j_r\leq n$ such that $\la_{j_k}=i+1$ or $-i$, for $1\leq k\leq r$,
with
$
({\rm E}^{(r)}_i \Delta(\la):\Delta(\la - \varepsilon_{j_1}-\cdots-\varepsilon_{j_r}))=2^r,
$

\item[(2)] ${\rm F}^{(r)}_i \Delta(\la)$ has a filtration whose non-zero subquotients consist of Verma modules
$\Delta(\la + \varepsilon_{j_1}+\cdots+\varepsilon_{j_r})$ for $1\leq j_1<\ldots<j_r\leq n$ such that $\la_{j_k}=i$ or $-i-1$, for $1\leq k\leq r$,
with
$
({\rm F}^{(r)}_i \Delta(\la):
\Delta(\la + \varepsilon_{j_1}+\cdots+\varepsilon_{j_r}))=2^r.
$
\end{itemize}
\end{lem}

\begin{proof}
Let us prove (2) only, since the proof for (1) is similar.

Let $\mu\in \La_\Z$ be a weight of $S^r(V)$.
The weight space $S^r(V)_\mu$ is a finite-dimensional $\mf h$-module, and its composition series has a unique composition factor $W_\mu$ with
\begin{equation}\label{eq:mult-1}
[S^r(V)_\mu:W_\mu]=\dim S^r(V)_\mu/\dim W_\mu.
\end{equation}
Also $W_\la\otimes W_\mu$ has a composition series with a unique composition factor $W_{\la+\mu}$, and
\begin{equation}\label{eq:mult-2}
\begin{split}
[W_\la\otimes W_\mu:W_{\la+\mu}]&=\dim(W_\la\otimes W_\mu)/\dim W_{\la+\mu}\\
&=2^{\lceil \ell(\la)/2 \rceil +\lceil \ell(\mu)/2 \rceil - \lceil \ell(\la+\mu)/2 \rceil}=\dim W_{\mu},
\end{split}
\end{equation}
since $\ell(\la)=\ell(\la+\mu)=n$.
Let $\nu_1,\ldots,\nu_N$ be an enumeration of weights of $S^r(V)$.
Then we have a filtration of $\h$-modules $0=U_0\subset U_1\subset \cdots \subset U_N=S^r(V)$ such that $U_p/U_{p-1}\cong S^r(V)_{\nu_p}$ for $1\leq p\leq N-1$.

Hence $W_\la\otimes S^r(V)$ has a composition series as an $\h$-module, where each composition factor is isomorphic to $W_{\la+\nu_p}$ for some $p$, and
\begin{equation}\label{eq:mult-3}
[W_\la\otimes S^r(V):W_{\la+\nu_p}]
=[S^r(V)_{\nu_p}:W_{\nu_p}][W_\la\otimes W_{\nu_p}:W_{\la+\nu_p}]=
\dim S^r(V)_{\nu_p},
\end{equation}
by \eqref{eq:mult-1} and \eqref{eq:mult-2}.
Now, applying the exact functor $U(\G)\otimes_{U(\mf b)} ? $ to this composition series of $W_\la\otimes S^r(V)$, we have a Verma flag of $\Delta(\la)\otimes S^r(V)$, where each non-zero subquotient is isomorphic to $\Delta(\la+\nu_p)$ for some $p$, and
\begin{equation*}\label{eq:mult-4}
\left(\Delta(\la)\otimes S^r(V) : \Delta(\la+\nu_i) \right)
=[W_\la\otimes S^r(V):W_{\la+\nu_p}]=\dim S^r(V)_{\nu_p},
\end{equation*}
by \eqref{eq:mult-3}.
Moreover, $\Delta(\la+\nu_p)$ occurs in ${\rm F}^{(r)}_i \Delta(\la)$ if and only if $\nu_p = \varepsilon_{j_1}+\cdots+\varepsilon_{j_r}$,
for some $1\leq j_1<\cdots<j_r\leq n$ such that $\la_{j_k}=i$ or $-i-1$ for $1\leq k\leq r$.

Let $\nu_p$ be such that  $\Delta(\la+\nu_p)$ occurs in ${\rm F}^{(r)}_i \Delta(\la)$.
Since
\begin{equation*}
S^r(V) \cong \bigoplus_{s=0}^r  S^s(\C^{n|0})\otimes \Lambda^{r-s}(\C^{0|n}),
\end{equation*}
as a vector space, the vectors
\begin{equation}\label{eq:super symmetric power}
(u_{\ov{i_1}}\ldots u_{\ov{i_s}})\otimes (u_{i_{s+1}}\wedge\ldots\wedge u_{i_r}),
\end{equation}
for $1\leq i_1<\cdots< i_s \leq n$ and $1\leq i_{s+1}< \cdots< i_r\leq n$ such that $\{\,i_1,\ldots,i_r\,\}=\{\,j_1,\ldots,j_r\,\}$, form a basis of $S^r(V)_{\nu_p}$, where the $\h_\even$-weight of the vector \eqref{eq:super symmetric power} is $\varepsilon_{i_1}+\cdots+\varepsilon_{i_r}=\varepsilon_{j_1}+\cdots+\varepsilon_{j_r}$. Hence we have
\begin{equation*}
\dim S^r(V)_{\nu_p}=2^r.
\end{equation*}
The proof is completed.
\end{proof}

Let
\begin{equation*}
\begin{split}
\TT^n_{\Z[q^{\pm}]}:=\bigoplus_{\la\in \La_\hZ}\Z[q,q^{-1}]M_\la,\quad \TT^n_\Z:=\Z\otimes_{\Z[q,q^{-1}]}\TT^n_{\Z[q^{\pm}]},
\end{split}
\end{equation*}
where $\Z$ is regarded as the right $\Z[q,q^{-1}]$-module with $q$ acting  as $1$.
If we put $\UU_\Z:=\Z\otimes_{\Z[q,q^{-1}]}\UU_{\Z[q^{\pm}]}$, then $\TT^n_\Z$ is a $\UU_\Z$-module since $\TT^n_{\Z[q^{\pm}]}$ is a $\UU_{\Z[q^{\pm}]}$-module.
Let $M_\la(1)=1\otimes M_\la\in \TT^n_\Z$, for $\la\in \La_\hZ$.
Define a $\Z$-linear isomorphism  by
\begin{align}
\label{Psi2}
\Psi:K(\mc O^\Delta_{n,\hZ}) \longrightarrow \TT^{n}_\Z, \qquad
 [\Delta (\la)] \mapsto M_{\la}(1)  \quad (\la\in\La_\hZ).
\end{align}
We claim that (up to some scalar multiples) the translation functors on the Grothendieck group of $\mc O^\Delta_{n,\hZ}$
act as the divided powers of the Chevalley generators in the quantum group of type $C$.

\begin{prop}\label{prop:translation on K}
The isomorphism  $\Psi:K(\mc O^\Delta_{n,\hZ}) \rightarrow \TT^{n}_\Z$ in \eqref{Psi2} sends
\begin{equation*}
\Psi([{\rm E}^{(r)}_iM]) =2^r E^{(r)}_i\Psi([M]),\quad
\Psi([{\rm F}^{(r)}_iM]) =2^r F^{(r)}_i\Psi([M]),
\end{equation*}
for $i\geq 0$, $r\geq 1$, and $M\in \mc O^{\Delta}_{n,\hZ}$.
\end{prop}

\begin{proof}
By Lemma \ref{lem:translation} and the exactness of translation functors, it suffices to check for $r=1$ and $M=\Delta(\la)$.
Note $\Psi([\Delta(\la)]) =M_\la(1)$.
The action of $E_i$ (and $F_i$, respectively) on the basis element $M_\la$ of  $\TT^n_\Z$ is clearly seen to be compatible with
the action of ${\rm E}_i$ (and  ${\rm F}_i$, respectively) given in Lemma~\ref{lem:translation}.
\end{proof}

\subsection{Translation functors on $\mc E_{n,\hZ}$ and $\mc F_{n,\hZ}$}

Note that $\mc E_{n,\hZ}$ is invariant under ${\rm E}_i$ and ${\rm F}_i$ for $i\geq 0$.

\begin{prop}\label{prop:translation on E}
The $\Z$-linear isomorphism   $\Psi_e:K(\mc E_{n,\hZ}) \rightarrow \EE^{n}_\Z$ in \eqref{Psi} sends
\begin{equation*}
\Psi_e([{\rm E}_iM]) =2 E_i\Psi_e([M]),\quad
\Psi_e([{\rm F}_iM]) =2 F_i\Psi_e([M]),
\end{equation*}
for $i\geq 0$ and $M\in \mc E_{n,\hZ}$.
\end{prop}
\begin{proof}
Let us consider the case of ${\rm F}_i$ only, as the proof for ${\rm E}_i$ is similar.
For $\la\in\La^+_\hZ$, we have by \eqref{eq:Euler character}
\begin{align*}
{\rm ch}E(\la)=&2^{\lceil n/2 \rceil}\sum_{w\in \mf S_n}w\left(
x_1^{\la_1}\cdots x_n^{\la_n} \prod_{i<j}\frac{x_i+x_j}{x_i-x_j}
\right),\\
=&\frac{2^{\lceil n/2 \rceil}}{D}\sum_{w\in \mf S_n}(-1)^{\ell(w)}w\left(
x_1^{\la_1}\cdots x_n^{\la_n}
\right),
\end{align*}
where $x_i=e^{\varepsilon_i}$ for $1\leq i\leq n$. Thus, by the Weyl character formula we have
\begin{equation}\label{eq:euler=prod of Schur}
{\rm ch}E(\la)=2^{\lceil n/2 \rceil} s_{\rho_n}(x_1,\ldots,x_n)s_{\la-\rho_n}(x_1,\ldots,x_n),
\end{equation}
where $\rho_n=(n-1,n-2,\ldots,1,0)$ and $s_\mu(x_1,\ldots,x_n)$ is the Laurent Schur polynomial in $x_1,\ldots,x_n$, for $\mu\in \La^{\geqslant}_\Z\cup \La^{\geqslant}_\hZ$.
Note that ${\rm ch}V=2(x_1+\cdots+x_n)=2 s_{(1,0,\ldots,0)}(x_1,\ldots,x_n)$.
It follows by the Littlewood-Richardson rule that
\begin{equation}\label{eq:LR rule}
s_{\la-\rho_n}(x_1,\ldots,x_n)s_{(1,0,\ldots,0)}(x_1,\ldots,x_n)=\sum_{\delta}s_{\delta}(x_1,\ldots,x_n),
\end{equation}
where the sum is over all $\delta\in\La^{\geqslant}_{\hZ}$ such that $\delta-(\la-\rho_n)=\varepsilon_r$ for some $1\leq r\leq n$. By putting $\nu=\delta+\rho_n$, we conclude from \eqref{eq:euler=prod of Schur} and \eqref{eq:LR rule} that
${\rm ch}(E(\la)\otimes V) =2 \sum_{\nu}{\rm ch}E(\nu)$,
where the sum is over all $\nu\in \La^+_\hZ$ such that $\nu=\la+\varepsilon_r$ for some $1\leq r\leq n$. By definition of ${\rm F}_i$, ${\rm ch}({\rm F}_i E(\la))$ is the partial sum of ${\rm ch}(E(\la)\otimes V)$ over $\nu$ when ${\rm wt}(\nu)={\rm wt}(\la)+\alpha_i$. By comparing with Lemma \ref{lem:action of U on E-basis}, we may conclude that
$\Psi_e([{\rm F}_i M]) = 2F_i\Psi_e([M])$, for $M\in \mc E_{n,\hZ}$.
\end{proof}

For $\la\in\La^+_{\hZ}$ let $U(\la)$ denote the injective hull of $L(\la)\in \mc E_{n,\hZ}$. Using co-induction as in \cite[Lemma 4.40]{Br2}, it follows that $U(\la)$ is finite-dimensional. Let $\mc F_{n,\hZ}$ denote the full subcategory of injective modules in $\mc E_{n,\hZ}$. Then every object in $\mc F_{n,\hZ}$ is a direct sum of finitely many $U(\la)$'s, and $\{\,[U(\la)]\,\vert\,\la\in\La^+_{\hZ}\,\}$ is a basis of $K(\mc F_{n,\hZ})$. The category $\mc F_{n,\hZ}$ is invariant under the translation functors ${\rm E}_i$ and ${\rm F}_i$, for $i\ge 0$.

We have a pairing $\langle\cdot,\cdot\rangle:K(\mc E_{n,\hZ})\times K(\mc F_{n,\hZ})\longrightarrow \Z$ by extending linearly the following formula (see \cite[Lemma 4.45]{Br2}):
\begin{equation*}\label{defn:pairing}
\langle [M],[N]\rangle:=\begin{cases}\dim_\C {\rm Hom}_\G(M^*,N), \text{ if }M\text{ is of type } \texttt{M},\cr
\hf\dim_\C {\rm Hom}_\G(M^*,N), \text{ if }M\text{ is of type } \texttt{Q},
\end{cases}
\end{equation*}
for $M\in \mc E_{n,\hZ}$ and $N\in \mc F_{n,\hZ}$.
We note that $L(\la)^*\cong L(-w_0\la)$, for all $\la\in\La^+$, and hence we have
\begin{align}\label{defn:pairing1}
\langle[L(-w_0\la)],[U(\mu)]\rangle = \delta_{\la\mu},\quad \text{ for $\la,\mu\in\La^+_\hZ$.}
\end{align}

Let
\begin{equation*}
\begin{split}
\FF^n_{\Z[q^{\pm}]}:=\bigoplus_{\la\in \La^+_\hZ}\Z[q,q^{-1}]U_\la,\quad \FF^n_\Z:=\Z\otimes_{\Z[q,q^{-1}]}\FF^n_{\Z[q^{\pm}]}.
\end{split}
\end{equation*}
We define a $\Z$-linear isomorphism by
\begin{align}\label{eq:map Phi}
\Phi: K(\mc F_{n,\Z})\longrightarrow \mathbb F^n_\Z,
\qquad
[U(\la)] \mapsto U_\la(1)\quad (\la\in\La^+_\hZ).
\end{align}
Recall from \eqref{eq:dual pairing} that we have a bilinear pairing $\langle\cdot,\cdot\rangle: \EE^n_{\Z[q^{\pm}]}\times\FF^n_{\Z[q^{\pm}]}\rightarrow\Z[q,q^{-1}]$, and let $\langle\cdot,\cdot\rangle_{q=1}$ be its specialization at $q=1$.
It follows from \eqref{defn:pairing1} that
\begin{align*}
\langle[M],[N]\rangle = \langle\Psi([M]),\Phi([N])\rangle_{q=1},\quad \text{for $M\in \mc E_{n,\hZ}$, $N\in \mc F_{n,\hZ}$. }
\end{align*}

\begin{prop}
The $\Z$-linear isomorphism $\Phi:K(\mc F_{n,\hZ}) \rightarrow \FF^{n}_\Z$ given by \eqref{eq:map Phi} sends
\begin{equation*}
\Phi([{\rm E}_iN]) =2 E_i\Phi([N]),\quad
\Phi([{\rm F}_iN]) =2 F_i\Phi([N]),
\end{equation*}
for $i\geq 0$ and $N\in \mc F_{n,\hZ}$.
\end{prop}

\begin{proof} We shall only prove the first identity, as the proof for the second is analogous.
Let $N\in\mc F_{n,\hZ}$. For every $M\in\mc E_{n,\hZ}$ we have
\begin{align*}
\langle\Psi_e([M]), \Phi([{\rm E}_i N])\rangle_{q=1}&=\langle[M],[{\rm E}_iN]\rangle = \langle [{\rm E}_iM],[N]\rangle = \langle\Psi_e([{\rm E_i}M]),\Phi([N])\rangle_{q=1}\\
&= \langle2 E_i\Psi_e([M]),\Phi([N])\rangle_{q=1} = \langle\Psi_e([M]),2 E_i\Phi([N])\rangle_{q=1}.
\end{align*}
The penultimate identity follows from Proposition \ref{prop:translation on E}, while the last identity follows from the fact that $\langle uv,w\rangle=\langle v,\sigma(u)w\rangle$ for $u,v\in\TT^n$ and $u\in\UU$, which can be established as \cite[Lemma 2.21]{Br2}. Now, since $M$ is arbitrary, we conclude that $\Phi({\rm E}_i[N])=2 E_i\Phi([N])$, as claimed.
\end{proof}

\subsection{Type $C$ Kazhdan-Lusztig conjecture for $\qn$}  
\label{subsec:KL:C}

Recall that the Kazhdan-Lusztig (or simply KL) conjecture for $\mc O_{n,\Z}$ (of $\qn$-modules of integer weights) was formulated in \cite[Section 4.8]{Br2}
(and recalled below as Conjecture \ref{conj:KL conj for integer}).
For the category $\mc O_{n,\hZ}$, we have the following conjecture.

 \begin{conj} [Conjecture $C$: a first version]
 \label{conj:half-integer}
The isomorphism  $\Psi:K(\mc O^\Delta_{n,\hZ}) \rightarrow \TT^{n}_\Z$ in \eqref{Psi2} sends
$[T(\la)]$ to a canonical basis element $T_\la(1):=1\otimes T_\la$, for $\la\in\La_\hZ$.
Equivalently, we have
\begin{equation*}\label{eq:char for tilting}
[T(\la)] = \sum_{\mu\preceq \la}t_{\mu \la}(1)[\Delta(\mu)].
\end{equation*}
\end{conj}

Since we have
\begin{equation*}
(T(\la) : \Delta(\mu)) = [\Delta(-\mu):L(-\la)],
\end{equation*}
by \cite[Theorem 6.4]{Br3}, Conjecture \ref{conj:half-integer} is equivalent to
\begin{equation}
\label{eq:ML}
[\Delta(\la)] = \sum_{\mu\preceq\la}t_{-\la -\mu}(1)[L(\mu)],\quad
 [L(\la)]=\sum_{\mu\preceq\la}\ell_{\mu\la}(1)[\Delta(\mu)],
\end{equation}
where the second formula is possibly an infinite sum.


\begin{rem}
S.~Tsuchioka \cite{Tsu} has informed us that the canonical bases  on the tensor product of the natural representations
of type $B$ and $C$ fail to have positivity when $n=4$ and $n=6$, respectively. That is, there are examples of the polynomials $t_{\mu \la}(q)$ defined in \eqref{can:dual:can} and \eqref{TLMN} that have negative coefficients (and moreover $t_{\mu \la} (1) \not \in \N$);
his computation was subsequently confirmed by Brundan.
Hence Conjecture~\ref{conj:half-integer} (and Brundan's original ``type $B$" conjecture in \cite{Br2} or Conjecture \ref{conj:KL conj for integer}) are incorrect.
Still we choose to keep this version of the conjecture (although incorrect), as it serves as a motivation for the ensuing discussion below.
\end{rem}

Nevertheless there are several crucial facts (such as linkage and type $B/C$ relations of translation functors on Grothendieck group level),
which support the relevance of the type $B/C$ canonical bases.

This is reminiscent of the breakdown of Lusztig conjecture on irreducible characters of algebraic groups over fields of not too large prime characteristics.
A corrected formulation of the tilting character solution for algebraic groups has been given by Riche-Williamson \cite{RW} in terms of the
$p$-canonical bases of affine Hecke algebras (which is proven in type $A$).
Note that the $p$-canonical basis of affine Hecke algebra is very much in the spirit of
the categorical canonical basis arising from Khovanov-Lauda-Rouquier  (KLR) categorification.
Therefore, it is natural to hope for a similar corrected KL conjecture in our superalgebra setting.

Recall  from \eqref{eq:V} that $\VV$ is the natural representation of the type $C$ quantum group $\UU$,
and that $\TT^n =\VV^{\otimes n}$ (with the $\Z[q^{\pm 1}]$-form $\TT^{n}_{\Z[q^{\pm 1}]}$).
Recall that Webster \cite{Web} has defined a categorification of a tensor product of  highest weight integrable modules of a quantum group (of type $C$).
In particular we have a $2$-representation $\mathfrak V^n$ of the KLR 2-category which categorifies the $\UU$-module $\VV^{\otimes n}$, that is, there is an isomorphism of
$\UU_{\Z[q^{\pm 1}]}$-modules  $\kappa: K_0(\mathfrak V^n) \stackrel{\cong}{\rightarrow} \TT^{n}_{\Z[q^{\pm 1}]}$. (Here $K_0$ denotes the
 split Grothendieck group.)
Denote by $\{ T_\la^{\mathfrak c} \vert \la \in \La_{\hZ} \}$ the {\em can$\oplus$nical} basis (or the orthodox basis in Webster's terminology)
of $K_0(\mathfrak V^n)$ consisting of the isoclasses of the self-dual indecomposable $1$-morphisms.
Note that Webster worked with quantum groups of type $C_n$ of finite rank, and $\mathfrak V^n$ or $K_0(\mathfrak V^n)$ here
arises naturally as a stable limit when the rank $n$ goes to infinity.
We denote the basis of $K_0(\mathfrak V^n)$ corresponding to $\{ M_\la \}$ in $\TT^{n}_\Z$ by  $\{ M_\la^{\mathfrak c} \}$, i.e., $\kappa ( M_\la^{\mathfrak c} ) =M_\la.$
Denote
\begin{equation}\label{can:dual:canC2}
\begin{split}
T_\la^{\mathfrak c}=M_\la^{\mathfrak c}+\sum_{\mu\prec\la}t_{\mu\la}^{\mathfrak c} (q) M_\mu^{\mathfrak c}.  
\end{split}
\end{equation}
It is known \cite{Web2} that $t_{\mu\la}^{\mathfrak c} (q) \in \N[q,q^{-1}]$. Denote $K_0(\mathfrak V^n)_\Z = \Z \otimes_{\Z[q^{\pm 1}]} K_0(\mathfrak V^n)$,
where $q$ acts on $\Z$ as $1.$ Recall the isomorphism $\Psi:K(\mc O^\Delta_{n,\hZ}) \rightarrow \TT^{n}_\Z$ in \eqref{Psi2}.

 \begin{conj} [Conjecture $C$]
 \label{conj:half-integer2}
The isomorphism  $\kappa^{-1}\circ \Psi:K(\mc O^\Delta_{n,\hZ}) \rightarrow K_0(\mathfrak V^n)_\Z$ sends
$[T(\la)]$ to a can$\oplus$nical basis element $T_\la^{\mathfrak c}(1) :=1\otimes T_\la^{\mathfrak c}$, for $\la\in\La_\hZ$.
Equivalently, we have
\begin{equation*}\label{eq:char for tilting}
[T(\la)] = \sum_{\mu\preceq \la}t_{\mu \la}^{\mathfrak c}(1)[\Delta(\mu)].
\end{equation*}
\end{conj}

Recall by Corollary~\ref{cor:ACB=CCB} that the type $C$ canonical basis of the following subspace
\begin{equation}
\label{TTs:min}
\TT^{n}_{\min} := \bigoplus_{l+m=n} \TT^{n}_{{\bf s}_{\min}(l,m)} \subset \TT^{n}
\end{equation}
are indeed type $A$ canonical basis (see \eqref{TTs} and \eqref{s:min} for notation). In particular, this set of canonical basis has positivity property, and we conjecture that they coincide with the corresponding set of Webster's can$\oplus$nical basis elements.

\begin{conj} [Conjecture~$C=A$]
 \label{conj:KLC=A}
The tilting characters in the full subcategory ${}_{\min} \mc O^\Delta_{n,\hZ}$ of $\mc O^\Delta_{n,\hZ}$ {\rm (}which correspond  to the subspace \eqref{TTs:min}{\rm )}
are solved by the  type $A$ canonical bases in Corollary~\ref{cor:ACB=CCB}.
\end{conj}

\subsection{Type $A$ Kazhdan-Lusztig conjecture for $\qn$}
\label{subsec:KL:A}

Recall from \cite{CMW} that
various category equivalences induced by twisting, odd reflection, and parabolic induction functor reduce the irreducible and tilting character problem in the BGG category of
$\mathfrak{gl}(m|n)$-modules of {\em arbitrary weights} to the setting of integer weights for which the Brundan's Kazhdan-Lusztig conjecture (and a theorem of \cite{CLW,BLW}) is applicable.
There are similar, though more complicated, category equivalences for $\G =\qn$, which are established by Chen \cite[Theorem 1.1]{Chen}.
By Chen's result, the problem of computing the characters of the irreducible and tilting modules of arbitrary weights in the BGG category for a queer Lie superalgebra  is reduced to the character problem
in the following three categories:
\begin{equation*}
\text{(i)}\ \ \mc O_{n,\Z} , \quad
\text{(ii)}\ \ \mc O_{n,\hZ},\quad
\text{(iii)}\ \ \mc O_{n,s^l}\ \  (s\in \C \setminus \tfrac{1}{2} \Z,\ l \in \{0,\ldots,n\}),
\end{equation*}
where $\mc O_{n,s^l}$ is the full subcategory of $\G$-modules in $\mc O_{n}$ with weights in
\begin{equation*}
\La_{s^l}:=\left\{\,\la=(\la_1,\ldots,\la_n)\in \mf h^*_\even\,  \Bigg\vert \,
\begin{array}{l}
\text{{\rm (1)} $\la_i \equiv s \!\mod \Z$ \quad for $1\le i \le l$,}       \\
\text{{\rm (2)} $\la_i \equiv -s \!\mod \Z$ \quad for $l+1\le i \le n$}       \\
\end{array}
\,\right\}.
\end{equation*}
Since the first two categories are covered by Conjecture \ref{conj:KL conj for integer} (see also Remark \ref{rem:remark on conj for integral weights}) and Conjecture \ref{conj:half-integer2} respectively, it remains to consider the category  $\mc O_{n,s^l}$.

Let $s\in \C \setminus \tfrac{1}{2} \Z$ and $l \in \{0,\ldots,n\}$.
Let $\VV_s :=\bigoplus_{i\in s+ \Z} \Q(q) v_i$ be the natural representation of $\UU_q(\mf{sl}_\infty)$, and let $\WW_s =\VV^*_s$ be its restricted dual. Note that
the uniform shift of the indices of $v_i$ by $s$ does not make a difference when we regard $\VV_s$ as a module over $\UU_q(\mf{sl}_\infty)$ instead of $\UU_q(\mf{gl}_\infty)$.
We should point out that $\UU_q(\mf{sl}_\infty)$ is different from $\UU^{\mf a}$ which is also of type $A$ but with one-sided infinity.
Set $m =n-l$. Consider the $\UU_q(\mf{sl}_\infty)$-module $\TT^{l|m}_s := \VV^{\otimes l}_s \otimes \WW^{\otimes m}_s$ with the standard monomial basis  $\{\,M_{\la}\,|\,\la\in \La_{s^l}\,\}$. We let the Bruhat ordering $\preceq_{\mf a}$ on $\La_{s^l}$ be the ``type $A$" Bruhat ordering as given in \cite[Section 2-b]{Br1}.
Then in a suitable completion of the $\UU_q(\mf{sl}_\infty)$-module $\TT^{l|m}_s$, there exist a canonical basis $\{\,T_{\la}\,|\,\la\in \La_{s^l}\,\}$ and
dual canonical basis $\{\,L_\la\,|\,\la\in \La_{s^l}\,\}$, which are bar invariant such that
\begin{equation}
T_\la=M_\la+\sum_{\mu\prec_{\mf a} \la}t_{\mu\la}^\aaf(q)M_\mu,\quad
L_\la=M_\la+\sum_{\mu\prec_{\mf a} \la}\ell_{\mu\la}^\aaf(q)M_\mu,
\end{equation}
for  $t_{\mu\la}^\aaf(q)\in q\Z[q]$ and $\ell_{\mu\la}^\aaf(q)\in q^{-1}\Z[q^{-1}]$.
We set $t_{\la\la}^\aaf(q)=\ell_{\la\la}^\aaf(q)=1$. 
Now, we make the following conjecture for $\mc O_{n,s^l}$.

\begin{conj}  [Conjecture~$A$]
\label{conj:KL conj for s-weights}
We have the following character formulae in the category $\mc O_{n,s^l}$: for $\la\in\La_{s^l}$, we have
$$
[T(\la)] = \sum_{\mu\preceq \la}t_{\mu \la}^\aaf(1)[\Delta(\mu)],
\quad
[L(\la)]=\sum_{\mu\preceq \la}\ell_{\mu\la}^\aaf(1)[\Delta(\mu)].
$$
\end{conj}

\begin{rem}
One may verify that the translation functors on the category $\mc O_{n,s^l}$ satisfy the Serre relations for the Lie algebra $\mf{sl}_\infty$.
Also the Bruhat ordering in $\mc O_{n,s^l}$ is consistent with the algebraic Bruhat ordering from the $\UU_q(\mf{sl}_\infty)$-module $\TT^{l|n-l}_s$.
\end{rem}

\appendix{}
\section{Category $\mc O_{n,\Z}$ and canonical bases}
\label{sec:notations for B}

Recall that $\mc O_{n,\Z}$  denotes the full subcategory of $\mc O_n$ consisting of integer weight modules.
In this Appendix, we review Brundan's KL conjecture for $\mc O_{n,\Z}$, and suggests an enhancement for it using canonical basis arising from the quantum covering group.

\subsection{Brundan's Kazhdan-Lusztig conjecture for $\mc O_{n,\Z}$}
\label{subsec:BGG for integer weights}

We first give a brief review on the results in \cite{Br2}. We first follow the settings and notations in Section \ref{sec:O} for representations of $\qn$.

Let $P^{\mf b}=\bigoplus_{a\in\Z_+}\Z \delta_a$ with a symmetric bilinear form $(\ep_a,\ep_b)=2\delta_{ab}$, for $a, b\in \Z_{+}$. Let $\mf b_\infty$ be the Kac-Moody Lie algebra of type $B$ of infinite rank with associated Dynkin diagram and simple roots given by

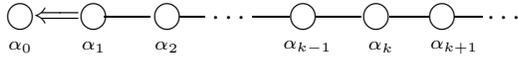
\begin{center}
\hskip -1cm  \setlength{\unitlength}{0.16in}
\begin{picture}(24,4)

\put(5.6,2){\makebox(0,0)[c]{$\bigcirc$}}
\put(8,2){\makebox(0,0)[c]{$\bigcirc$}}
\put(10.4,2){\makebox(0,0)[c]{$\bigcirc$}}
\put(14.85,2){\makebox(0,0)[c]{$\bigcirc$}}
\put(17.25,2){\makebox(0,0)[c]{$\bigcirc$}}
\put(19.4,2){\makebox(0,0)[c]{$\bigcirc$}}
\put(8.35,2){\line(1,0){1.5}} \put(10.82,2){\line(1,0){0.8}}
\put(13.2,2){\line(1,0){1.2}} \put(15.28,2){\line(1,0){1.45}}
\put(17.7,2){\line(1,0){1.25}} \put(19.81,2){\line(1,0){0.9}}
\put(6.8,1.97){\makebox(0,0)[c]{$\Longleftarrow$}}
\put(12.5,1.95){\makebox(0,0)[c]{$\cdots$}}
\put(21.5,1.95){\makebox(0,0)[c]{$\cdots$}}
\put(5.6,1){\makebox(0,0)[c]{\tiny ${\alpha}_0$}}
\put(8,1){\makebox(0,0)[c]{\tiny ${\alpha}_1$}}
\put(10.4,1){\makebox(0,0)[c]{\tiny ${\alpha}_2$}}
\put(15,1){\makebox(0,0)[c]{\tiny ${\alpha}_{k-1}$}}
\put(17.4,1){\makebox(0,0)[c]{\tiny ${\alpha}_k$}}
\put(19.8,1){\makebox(0,0)[c]{\tiny ${\alpha}_{k+1}$}}
\end{picture}
\end{center} \vskip -5mm
\begin{equation*}
\begin{split}
& \alpha_0=-\delta_{1} ,\ \ \alpha_i=\delta_{i}-\delta_{i+1} \ (i\geq 1).
\end{split}
\end{equation*}
Let $t$ be an indeterminate, and let $\UU^{\mf b}=\UU_t(\mf b_\infty)$ be the associated quantum group defined as in \eqref{eq:defining relations for QG} over $\Q(t)$.
For $n\geq 1$, we put
\begin{equation}\label{eq:Lambda over integer}
\begin{split}
&\La_{\Z}:=\{\,\la=(\la_1,\ldots,\la_n)\,|\,\la_i\in\Z, \ \forall i\,\},\\
&\La^{\geqslant}_{\Z}:=\{\,\la \in \La_{\Z}\,|\,\la_1\ge\la_2\ge\cdots\ge\la_n\,\},\\
&\La^{+}_{\Z}:=\{\,\la \in\La^{\geqslant}_{\Z}\,|\, \la_i=\la_{i+1}\text{ implies }\la_i=0, \ \forall i \,\},\\
&\La^{\times}_{\Z}:=\{\,\la\in \La_{\Z}\,|\,\la_i\neq 0, \ \forall i \,\},\\
&\La^{+,\times}_{\Z}:=\La^{+}_{\Z}\cap \La^{\times}_{\Z}.
\end{split}
\end{equation}
Let $\succeq$ be the Bruhat ordering on $\Lambda_\Z$ \cite[Section 2.3]{Br2}, which is defined similarly as in Section \ref{subsec:bruhat ordering} with respect to the simple roots of ${\mf b}_\infty$ in $P^{\mf b}$, and let $\succcurlyeq$  the one on $\La^{\geqslant}_\Z$ given in \cite[Lemma 2.1]{PS2}.
For $\la, \mu\in \La^{\geqslant}_{\Z}$, we have $\la\succeq \mu$ if and only if $\la\succcurlyeq \mu$ \cite[Proposition 3.3]{CK}.

Let $\VV^{\mf b}=\bigoplus_{a\in\Z}\Q(t)v_a$ be the natural representation of $\UU^{\mf b}$. For $n\geq 1$, let $\TT^{n,{\mf b}}:= (\VV^{\mf b})^{\otimes n}$ with $\Q(t)$-bases $\{\,N_\la:=v_{\la_1}\otimes \cdots\otimes v_{\la_n}\,|\,\la\in \La_\Z\,\}$ and $\{\,M_\la:=(t+t^{-1})^{-z(\la)}N_\la\,|\,\la\in\La_\Z\,\}$, where $z(\la)$ is the number of zero parts in $\la$.
Define $\widehat{\TT}^{n,{\mf b}}$ to be the completion of $\TT^{n,{\mf b}}$ (on which a bar involution is defined) as described in Sections \ref{subsec:completion} and \ref{subsec:canonical basis of Tn}.

Let $\{\,T_\la\,|\,\la\in\La_\Z\,\}$ and $\{\,L_\la\,|\,\la\in\La_\Z\,\}$ denote the {\em canonical basis} and {\em dual canonical basis}, respectively, which are the unique bar-invariant topological $\Q(t)$-bases of $\widehat{\TT}^{n,{\mf b}}$ such that
\begin{align}
\label{TLMN}
T_\la=\sum_{\mu}t_{\mu\la}(t)N_\mu,\quad
L_\la=\sum_{\mu}\ell_{\mu\la}(t)M_\mu,
\end{align}
with $t_{\la\la}(t)=\ell_{\la\la}(t)=1$, $t_{\mu\la}(t)\in t\Z[t]$, $\ell_{\mu\la}(t)\in t^{-1}\Z[t^{-1}]$ for $\mu\prec\la$, and $t_{\mu\la}(t)=\ell_{\mu\la}(t)=0$ otherwise  \cite[Theorem 2.22]{Br2}.

\begin{rem}\label{rem:completions}{\rm The completion  $\widehat{\TT}^{n,\mf b}$ here are different from the one in \cite{Br2}. More precisely, our completion is defined following the construction \cite[Section 3.2]{CLW} so called {\rm B-completion}, and it can be regarded as a proper subspace of the Brundan's completion.
For example, $\sum_{k\geq 1}M_{(0,k)}$ belongs to not $\widehat{\TT}^{n,\mf b}$ but   the completion in \cite{Br2}.
Hence by the uniqueness of the (dual) canonical bases, we may identify $\{\,T_\la\,|\,\la\in\La_\Z\,\}$ and $\{\,L_\la\,|\,\la\in\La_\Z\,\}$ with those in \cite{Br2}.}
\end{rem}

The following KL conjecture for the BGG category $\mc O_{n,\Z}$  (of integer weight modules)
is due to Brundan \cite[Section 4.8]{Br2}.
\begin{conj}
\label{conj:KL conj for integer}
For $\la\in \La_{\Z}$, we have
\begin{equation*}
[T(\la)] = \sum_{\mu\preceq \la}t_{\mu \la}(1)2^{z(\mu)}[\Delta(\mu)],
\quad
[L(\la)]=\sum_{\mu\preceq\la}\ell_{\mu\la}(1)[\Delta(\mu)].
\end{equation*}
\end{conj}

\begin{rem}\label{rem:remark on conj for integral weights}
For the same reason as remarked in Section~\ref{subsec:KL:C}, Brundan's Conjecture~\ref{conj:KL conj for integer} is incorrect thanks to \cite{Tsu}.
One may hope that a spin/super version of Webster's tensor product $2$-category (see \cite{KKT} for a spin/super quiver Hecke algebra) might provide the right can$\oplus$nical basis
$T^{\mathfrak b}_\la$ which correspond to the tilting module $T(\la)$. Actually the Grothendieck group of such a $2$-category should be a module of
a {\em quantum covering group} over $\Z[q^{\pm 1}, \pi]$, with $\pi^2=1$ (see \cite{CHW, C}). The quantum covering group specializes at $\pi=1$ to a quantum group of type $B$.
One would hope to have a $2$-parameter (i.e., $q$ and $\pi$) enhancement of the Kazhdan-Lusztig theory and its homological interpretation.
\end{rem}

Let $\FF^{n,{\mf b}}:={\TT}^{n,{\mf b}}/{\KK}^{n,{\mf b}} $ be the $n$th $t$-wedge space of $\VV^{\mf b}$ of type $B$ \cite[Section 3.1]{Br2} (cf.~\cite{JMO}) with $\Q(t)$-basis $\{\,F_\la:=\pi(N_{w_0\la}) \,|\,\la\in\La^{+}_{\Z}\,\}$, where $\pi : \TT^{n,{\mf b}} \rightarrow \FF^{n,{\mf b}}$ is the canonical projection. The bar involution on $\widehat{\TT}^{n,{\mf b}}$ induces a bar involution on the completion $\widehat{\FF}^{n,{\mf b}}$ of ${\FF}^{n,{\mf b}}$, and let $\{\,U_\la\,|\,\la\in\La^{+}_{\Z}\,\}$ be its associated unique bar-invariant topological $\Q(t)$-basis of $\widehat{\FF}^{n,{\mf b}}$, called the {\em canonical basis},  such that
\begin{align}\label{eq:CB on Fock B}
U_\la=\sum_{\mu\in\La^{+}_{\Z}} u_{\mu\la}(t) F_\mu,
\end{align}
with $u_{\la\la}(t)=1$, $u_{\mu\la}(t)
\in t\Z[t]$ for $\mu\succ\la$, and $u_{\mu\la}(t)=0$ otherwise \cite[Theorem 3.5]{Br2}. We have $\pi(T_{w_0\la})=U_\la$ if $\la\in \La^{+}_{\Z}$, and $0$ otherwise.
We put
\begin{equation*}
\EE^{n,{\mf b}}:=\bigoplus_{\la\in\La^{+}_{\Z}}\mathbb{Q}(t) E_\la \subset \widehat{\TT}^{n,{\mf b}},
\end{equation*}
where
\begin{equation}\label{eq:dual CB E B}
E_\la:=\sum_{\substack{\mu\in\La^{+}_{\Z}, \mu\preceq \la}}u_{-w_0\la,-w_0\mu}(t^{-1}) L_\mu,\quad \text{ for $\la\in\La^{+}_{\Z}$}.
\end{equation}
It is shown \cite[Theorem 3.16]{Br2} that $\{\,L_\la\,|\,\la\in\La^{+}_{\Z}\,\}$ is the unique bar-invariant basis of $\EE^{n,{\mf b}}$ such that $L_\la\in E_\la + \sum_{\mu\in\La^+_\Z}t^{-1}\Z[t^{-1}]E_\mu $, and
\begin{equation}\label{eq:dual CB L B}
L_\la=\sum_{\substack{\mu\in\La^{+}_{\Z}, \mu\preceq \la}}\ell_{\mu\la}(t) E_\mu.
\end{equation}
Put
\begin{equation*}
\begin{split}
&\EE^{n,\mf b}_{\Z[t^{\pm}]}:=\sum_{\la\in \La^+_\Z}\Z[t,t^{-1}]L_\la=\sum_{\la\in \La^+_\Z}\Z[t,t^{-1}]E_\la,\\
&\EE^{n,\mf b}_{\Z}:=\Z\otimes_{\Z[t,t^{-1}]}\EE^{n,\mf b}_{\Z[t^{\pm}]},
\end{split}
\end{equation*}
where $\Z$ is the right $\Z[t,t^{-1}]$-module with $t$ acting on $\Z$ as $1$.  Let $E_\la(1)=1\otimes E_\la$ and $L_\la(1)=1\otimes L_\la\in \EE^{n,\mf b}_{\Z}$, for $\la\in \La^+_\Z$.
Then we have the following.

\begin{thm}[Theorem 4.52 in \cite{Br2}]\label{thm:Brundan Main} Let {\rm $\Psi_e :K(\mc E_{n,\Z}) \rightarrow \EE^{n,\mf b}_\Z$} be the $\Z$-linear isomorphism defined by
$\Psi_e([E (\la)])= E_{\la}(1)$, for $\la\in\La^+_\Z.$
Then $\Psi_e \left([L(\la)]\right)={L}_{\la}(1)$, for $\la\in\La^+_\Z$.
\end{thm}

It follows immediately from  \eqref{eq:dual CB E B}, \eqref{eq:dual CB L B}, and Theorem \ref{thm:Brundan Main} that for all $\la\in\La^+_\Z$
\begin{equation*}
\begin{split}
[E(\la)]
&=\sum_{\mu\in\La^{+}_{\Z},\, \mu\preceq \la}u_{-w_0\la,-w_0\mu}(1) [L(\mu)],\\
[L(\la)]
&=\sum_{\mu\in\La^{+}_{\Z},\, \mu\preceq \la}\ell_{\mu\la}(1) [E(\mu)].
\end{split}
\end{equation*}

\subsection{Canonical basis of the $q$-wedge spaces of types $A$ and $B$}

Similar to Section~ \ref{subsec:CBA=CBC}, we shall show that the canonical basis $\{\,U^{\mf a}_\la\,|\, \la\in\La_{\hZ}^+,\ {\rm sgn}(\la)\in \Sigma_{l,m}\,\}$ of type $A$ in Section \ref{subsec:CBA=CBC} coincides with part of the canonical basis of type $B$ on the $q$-wedge space.

Recall that $\UU^{\mf b}$ is an algebra over $\Q(t)$ for an indeterminate $t$, while $\UU$ is over $\Q(q)$. From now on, let us assume $q=t^2$ or $t=q^{\hf}$. Since the $\Q(q)$-subalgebra of $\UU^{\mf b}$ generated by $E_i, F_i$ for $i\geq 1$ is isomorphic to $\UU^{\mf a}\subset \UU$ as a $\Q(q)$-algebra, we may regard that $\UU^{\mf a}\subset \UU^{\mf b}$.

Fix $n\ge 1$ and let $l, m\in\Z_+$ such that $l+m=n$. Let
\begin{equation*}
\FF^{\mf b}_{(l,m)}:=\bigoplus_{\substack{\la\in\La_{\Z}^{+,\times} \\ {\rm sgn}(\la)\in \Sigma_{l,m}}}\Q(q)F_\la,
\end{equation*}
where we define ${\rm sgn}(\la)$ as in \eqref{eq:sgn(lambda)}, and let $\widehat{\FF}^{\mf b}_{(l,m)}$ be its completion in $\widehat{\FF}^{n,{\mf b}}$.

\begin{lem}\label{lem:CB U on F(l,m)}
The set $\{\,U_\la\,|\, \la\in\La_{\Z}^{+,\times},\ {\rm sgn}(\la)\in \Sigma_{l,m}\,\}$ is a topological $\Q(q)$-basis of $\widehat{\FF}^{\mf b}_{(l,m)}$. In particular, we have $u_{\mu\la}(t)\in \Z[t^2]=\Z[q]$.
\end{lem}

\begin{proof}
First we see that  $\FF^{\mf b}_{(l,m)}$ is a $\UU^{\mf a}$-submodule over $\Q(q)$ (cf.~\cite[Section 3.1]{Br2}). We define a bar involution $\psi_{\mf a}$ on $\widehat{\FF}^{{\mf b}}_{(l,m)}$ as in $\widehat{\FF}_{(l,m)}$ over $\Q(q)$. Then we have a $\psi_{\mf a}$-invariant canonical basis of $\widehat{\FF}^{{\mf b}}_{(l,m)}$, which is clearly equal to a $\psi_{\mf a}$-invariant canonical basis of $\Q(t)\otimes_{\Q(q)}\widehat{\FF}^{{\mf b}}_{(l,m)}$.

On the other hand, let $\psi$ denote the bar involution on $\widehat{\FF}^{n, \mf b}$ over $\Q(t)$. By \eqref{eq:CB on Fock B} and \cite[Proposition 3.3]{CK}, we have ${\rm sgn}(\mu)\in \Sigma_{l,m}$ for $\mu\in\La^+_\Z$ with $\mu\succ \la$ so that $U_\la \in \Q(t)\otimes_{\Q(q)}\widehat{\FF}^{\mf b}_{(l,m)}$.

Note that as a module over $\UU^{\mf a}_{\Q(t)}:=\Q(t)\otimes_{\Q(q)} \UU^{\mf a}$, we have $\VV^{\mf b}=\VV^{\mf b}_+\oplus \Q(t)v_0 \oplus \VV^{\mf b}_-$, where
$
\VV^{\mf b}_\pm:=\bigoplus_{a\in\N}\Q(t) v_{\pm a}.
$
 Note that $\VV^{\mf b}_+$ is a natural representation of $\UU^{\mf a}_{\Q(t)}$, and $\VV^{\mf b}_-$ is its restricted dual, respectively.
By the same argument as in the proof of Proposition \ref{thm:head:can:basis} and Corollary \ref{cor:ACB=CCB} where we replace $\VV_\pm$ with $\VV^{\mf b}_\pm$, we can prove that the $\psi_{\mf a}$-invariant canonical basis and the $\psi$-invariant canonical basis of $\Q(t)\otimes_{\Q(q)}\widehat{\FF}^{\mf b}_{(l,m)}$ coincide.

Therefore $\{\,U_\la\,|\, \la\in\La_{\Z}^{+,\times},\ {\rm sgn}(\la)\in \Sigma_{l,m}\,\}$ is the $\psi_{\mf a}$-invariant canonical basis of $\widehat{\FF}^{\mf b}_{(l,m)}$. In particular, this implies that $u_{\mu\la}(t)\in \Z[t^2]=\Z[q]$.
\end{proof}

Define a bijection
\begin{equation*}\label{sharp:map}
\begin{split}
\La_\hZ\quad & \stackrel{\sharp}{\longrightarrow} \quad \La^\times_{\Z},\\
\la=(\la_1,\ldots,\la_n)& \ {\longmapsto} \
\la^\sharp:=\left(\la_1+ {\rm s}_1\tfrac{1}{2},\ldots,\la_n+ {\rm s}_n\tfrac{1}{2}\right),
\end{split}
\end{equation*}
where ${\rm sgn}(\la)=({\rm s}_1,\ldots,{\rm s}_n)$.
The map also induces a bijection $\sharp : \La^{+}_\hZ \rightarrow \La^{+,\times}_\Z$, which preserves the Bruhat ordering $\succeq$ \cite[Lemma 3.2]{CK}. Recall that ${\FF}^{\mf a}_{(l,m)}$ is a $\UU^{\mf a}$-module with basis $\{\,F^{\mf a}_\la\,|\,\la\in\La_{\hZ}^+, \ {\rm sgn}(\la)\in \Sigma_{l,m}\,\}$ (see \eqref{eq:Fock space A}). As in Lemma \ref{lem:Fa=F}, we have the following.

\begin{lem}\label{lem:Fa=Fb}
There exists an isomorphism of $\UU^{\mf a}$-modules  $\FF_{(l,m)}^{\mf a} \rightarrow \FF^{\mf b}_{(l,m)}$
which sends $F^{\mf a}_\la$ to $F_{\la^\sharp}$, for $\la\in\La_{\hZ}^+$ with ${\rm sgn}(\la)\in \Sigma_{l,m}$.
\end{lem}

Now, we define a $\Q(q)$-linear isomorphism
$
\rho^\sharp : \widehat{\FF}^{n,\mf a} \longrightarrow \widehat{\FF}^{n,\mf b}
$
by $\rho^\sharp(F^{\mf a}_\la)=F_{\la^\sharp}$, for $\la\in\La_{\hZ}^+$.

\begin{prop}\label{thm:ACB=BCB}
We have
$
\rho^\sharp(U^{\mf a}_\la) = U_{\la^\sharp},
$ 
for $\la\in\La_{\hZ}^+$. In particular, we have $u^{\mf a}_{\mu\la}(q)=u_{\mu^\sharp\la^\sharp}(t)$, for $\la, \mu \in\La_{\hZ}^+$.
\end{prop}

\begin{proof}
By Lemma \ref{lem:Fa=Fb}, we can show  by similar arguments as in Theorem \ref{thm:ACB=CCB} that
$\{\,\rho^\sharp(U^{\mf a}_\la)\,|\,\la\in\La_{\hZ}^+,\ {\rm sgn}(\la)\in \Sigma_{l,m}\,\}$ is a $\psi_{\mf a}$-invariant canonical basis of $\widehat{\FF}^{\mf b}_{(l,m)}$. On the other hand, we showed in the proof of Lemma \ref{lem:CB U on F(l,m)} that $\{\,U_\la\,|\, \la\in\La_{\Z}^{+,\times},\ {\rm sgn}(\la)\in \Sigma_{l,m}\,\}$ is also a $\psi_{\mf a}$-invariant  canonical basis of $\widehat{\FF}^{\mf b}_{(l,m)}$. Therefore $\rho^\sharp(U^{\mf a}_\la) = U_{\la^\sharp}$ for $\la\in\La_{\hZ}^+$ with ${\rm sgn}(\la)\in \Sigma_{l,m}$ by the uniqueness.
\end{proof}

\begin{cor}\label{cor:dual CB for B = dual CB for C}
For $\la, \mu\in\La_{\hZ}^+$, we have
$
u_{\mu\la}(t^2) = u_{\mu^\sharp\la^\sharp}(t),$ and $\ell_{\mu\la}(t^2)=\ell_{\mu^\sharp\la^\sharp}(t).
$
\end{cor}
\begin{proof}
The first identity follows immediately from Theorems \ref{thm:ACB=CCB} and \ref{thm:ACB=BCB}.
Consider the matrices $(u_{\mu\la}(q))$, where the indices $\la, \mu\in\La_{\hZ}^+$ are arranged with respect to a total ordering compatible with the Bruhat ordering. By \eqref{formula:U}, $(\ell_{\mu\la}(q))$ is the inverse of $(u_{\mu\la}(q))$. Similarly, by \eqref{eq:dual CB L B} $(\ell_{\mu^\sharp\la^\sharp}(t))$ is the inverse of
$(u_{\mu^\sharp\la^\sharp}(t))$, where the indices $\la^\sharp, \mu^\sharp\in\La_{\Z}^{+,\times}$ are arranged with respect to a total ordering induced from the bijection $\sharp$. Therefore the second identity follows from the first one.
\end{proof}



\end{document}